\documentclass{amsart}
\usepackage{amsmath,amsthm}
\usepackage{amsfonts,amssymb}

\usepackage{graphicx}
\usepackage{ifpdf}

\newtheorem{thm}{Theorem}[section]

\newtheorem{lem}[thm]{Lemma}
\newtheorem{prop}[thm]{Proposition}

\newtheorem{defn}[thm]{Definition}

 \usepackage{ifpdf}
    \numberwithin{equation}{section}
    \numberwithin{table}{section}
    \numberwithin{figure}{section}

\theoremstyle{remark}

\makeatletter
\newcommand{\bddots}{%
  \mathinner{\mkern1mu\raise\p@\vbox{\kern7\p@\hbox{.}}\mkern2mu
    \raise4\p@\hbox{.}\mkern2mu\raise7\p@\hbox{.}\mkern1mu}}
\makeatother

\def\a{{\alpha}}

\def\t{{\theta}}

\def\eb{{\mathbf e}}
\def\m{{\mathbf m}}

\def\ib{{\mathbf i}}
\def\jb{{\mathbf j}}
\def\kb{{\mathbf k}}
\def\lb{{\mathbf l}}
\def\mb{{\mathbf m}}
\def\sb{{\mathbf s}}
\def\tb{{\mathbf t}}

 \def\NN{{\mathbb N}}
 \def\CA{{\mathcal A}}
 
 \def\CG{{\mathcal G}}
 
 \def\CP{{\mathcal P}}
 \def\CK{{\mathcal K}}
 \def\CI{{\mathcal I}}
 \def\CL{{\mathcal L}}
 \def\CS{{\mathcal S}}
 \def\CT{{\mathcal T}}
 \def\CTC{{\mathcal T\!C}}
 \def\CTS{{\mathcal T\!S}}
 
 \def\CH{{\mathcal H}}

 \def\HH{{\mathbb H}}

 \def\RR{{\mathbb R}}
 
 \def\ZZ{{\mathbb Z}}
 \def\A{{\mathcal A}}
 \def\G{{\mathcal G}}
 
\newcommand{\e}{\mathrm{e}}

\newcommand{\tr}{{\mathsf {tr}}}
\newcommand{\TC}{{\mathsf {TC}}}
\newcommand{\TS}{{\mathsf {TS}}}

\def \la {\langle}
\def \ra {\rangle}

\newcommand{\wh}{\widehat}

 \ifpdf
 \DeclareGraphicsExtensions{.pdf,.png,.jpg}
 \else
 \DeclareGraphicsExtensions{.eps}
 \fi
 \graphicspath{{images/}}

\begin{document}

\title[Discrete Fourier analysis on dodecahedron and tetrahedron]
{Discrete Fourier analysis on a dodecahedron
and a tetrahedron}
\author{Huiyuan Li}
\address{Institute of Software\\
Chinese Academy of Sciences\\ Beijing 100080,China}
\email{hynli@mail.rdcps.ac.cn}
\author{Yuan Xu}
\address{  Department of Mathematics\\ University of Oregon\\
    Eugene, Oregon 97403-1222.}
\email{yuan@math.uoregon.edu}

\date{\today}
\keywords{Discrete Fourier series, trigonometric, Lagrange interpolation,
dodecahedron, tetrahedron}
\subjclass{41A05, 41A10}
\thanks{The first authors were supported by NSFC Grant 10601056,
10431050 and 60573023. The second author was supported by NSF
Grant DMS-0604056}

\begin{abstract}
A discrete Fourier analysis on the dodecahedron is studied, from which
results on a tetrahedron is deduced by invariance. The results include
Fourier analysis in trigonometric functions, interpolation and cubature
formulas on these domains. In particular, a trigonometric Lagrange
interpolation on the tetrahedron is shown to satisfy an explicit compact
formula and the Lebesgue constant of the interpolation is shown to be in
the order of $(\log n)^3$.
\end{abstract}

\maketitle

\section{Introduction}
\setcounter{equation}{0}

It is well known that Fourier analysis in several variables can be developed based
on the periodicity defined by a lattice, which is a discrete subgroup defined by
$A\ZZ^d$, where $A$ is a nonsingular $d\times d$ matrix. A lattice $L:= A\ZZ^d$
is called a tiling lattice of $\RR^d$ if there is bounded set $\Omega$ that tiles
$\RR^d$ in the sense that $\Omega + L =\RR^d$. Let $L$ be a tiling lattice and
$L^\perp:= A^{-\tr} \ZZ^d$ be its dual lattice; then a theorem of Fuglede \cite{F}
states that the family of exponentials $\{e^{2\pi i \alpha \cdot x }: \alpha \in L^\perp\}$
forms an orthonormal basis for $L^2(\Omega)$. The Fourier expansion on
$\Omega$ is essentially the usual multivariate Fourier series under a change of
variables $x \mapsto A^{-1} x$.

One can also develop a discrete Fourier analysis associated with a
lattice, starting with a discrete Fourier transform based on
$L^\perp$, which has applications in areas such as signal processing
and sampling theory (see, for example, \cite{DM, Hi, Ma}).  Recently
in \cite{LSX}, we studied the discrete Fourier transform and used it
to derive results on cubature and trigonometric interpolation on the
domain $\Omega$, both are important tools in numerical computation
and approximation theory. The simplest domain for the tiling lattice
is the regular hexagon, which has the invariance of the reflection
group $\CA_2$. The fundamental domain of the hexagon under $\CA_2$
is an equilateral triangle. A detailed study of the discrete Fourier
analysis is carried out on the hexagon and on the triangle in
\cite{LSX}. The invariant and the anti-invariant projections of the
basic exponential functions are analogues of cosine and sine
functions on the triangle, which have been studied previously in
\cite{K, Sun}. Explicit and compact formulas are derived for several
cubature formulas and interpolation functions in \cite{LSX}. In
particular, we found a compact formula for the Lagrange
interpolation by trigonometric functions that interpolates at
$X_n:=\{(\frac{i}{n}, \frac{j}{n}): 0 \le i \le j \le n\}$ on the
triangle $T: = \{(x,y): x,y \ge 0, \ x+y \le 1\}$ and proved that
its Lebesgue constant is in the order of $(\log n )^2$. The result
on interpolation is noteworthy since it is in sharp contract to the
algebraic polynomial interpolation on $X_n$, which has an
undesirable convergence behavior.

The purpose of the present paper is to carry out a similar analysis on $\RR^3$
for a tetrahedron, also called a simplex in $\RR^3$. For this we work with the
face-centered cubic (fcc) lattice, which has the symmetry of reflection group
$\CA_3$. The domain that tiles $\RR^3$ with the fcc lattice is the rhombic
dodecahedron (see Figure 3.2), whose fundamental domain under $\CA_3$
is a regular tetrahedron. We shall develop in detail a Fourier analysis on these
two domains, study analogues of cosine and sine functions as in the case of
hexagon, and establish compact formulas for discrete inner product, cubature
formulas and Dirichlet kernels.

Just as in the case of the regular hexagon
\cite{LSX,Sun}, the analysis on the rhombic dodecahedron and the tetrahedron
 is carried out using {\it homogeneous coordinates} of $\RR^4$ instead of in
 $\RR^3$. This has the advantage that our formulas are more symmetric and
 the symmetry of the domain becomes more transparent. The Fourier transform
 on the dodecahedron as well as the generalized cosine and sine functions
 were studied earlier in \cite{Sun2} using a homogeneous coordinate system
 in $\RR^6$. We choose our homogeneous coordinates in $\RR^4$ since
 $\CA_3$ can be regarded as a permutation group on four elements.

It should be pointed out that the development on a specific domain does not
follow immediately from the general theory. To tile the space without overlap,
the domain $\Omega$ can only include part of its boundary. For the discrete
Fourier analysis, this fact causes a loss of symmetry; for example, for the
rhombic dodecahedron, the discrete Fourier transform is defined using only
part of the boundary points. In order to obtain results that are symmetric, we have
to modify definitions to include all boundary points, which can be delicate if the
orthogonality is to be preserved. The difficulty lies in the congruent relations
of the boundary. In order to understand the periodicity based on the rhombic
dodecahedron, we need to understand the congruence of the boundary under
translation by the lattice. Furthermore, in order to transform results from the
rhombic dodecahedron to the tetrahedron, we need to understand the action
of $\CA_3$ on the boundary. The complication of the congruence of the
boundary is also one of the main reasons why we restrict ourself to $\RR^3$
instead of dealing with lattices on $\RR^d$ that are invariant under $\CA_d$
(see \cite{CS}) for all $d \ge 3$.

One of our main results is a compact formula for the Lagrange
interpolation based on the regular points on the tetrahedron, whose
Lebesgue constant is shown to be in the order of $(\log n)^3$.
Again, this is a result in sharp contrast to interpolation by
algebraic polynomials on the same set of points. Interpolation by
simple functions is an important tool in numerical analysis that has
a variety of applications. For interpolation on the point sets in
several variables, little results are known if the point sets are
not of tensor product type. Moreover, most studies consider mainly
interpolation by algebraic polynomials, which face the problem of
choosing interpolation points, as equally spaced points do not yield
favorable results \cite{Bos}. Our study in \cite{LSX} and in the
present paper demonstrates that interpolation at equally spaced
points on the triangle and on the tetrahedron can be solved with
{\it trigonometric functions}: the interpolation can be carried out
by compact formulas that offer fast computation, and the convergence
behavior is as good as can be expected since a Lebesgue constant of
$(\log n)^d$ for interpolation in $\RR^d$ is about optimal. In the
present paper, we concentrate on theoretic framework that leads to
trigonometric interpolation on the tetrahedron, numerical  study is
left out for a future work.

The paper is organized as follows. In Section 2 we sum up results from the
general theory of discrete Fourier analysis associated with lattice.
The analysis on the rhombic dodecahedron will be carried out in Section 3,
including a detailed study on the congruence of the boundary. Results on the
tetrahedron are developed in Section 4, including generalized sine and cosine
functions.


\section{Discrete Fourier analysis with Lattice}
\setcounter{equation}{0}

In this section we recall results on discrete Fourier analysis associated with
lattice. Background and the content of the first subsection can be found in
\cite{CS, DM, Hi, Ma}. Results in the second subsection are developed in
\cite{LSX}. We shall be brief and refer the proof and discussions to the
above mentioned references.


\subsection{ Lattice and Fourier series}

A  lattice $L$ of $\RR^d$ is a discrete subgroup that contains $d$ linearly
independent vectors,
\begin{align*}
   L:=\left\{ k_1 a_1 +k_2a_2+\cdots+k_da_d:\ k_i \in \ZZ,\ i=1,2,\cdots,d \right\},
 \end{align*}
where $a_1,\cdots,a_d$ are linearly independent column vectors in $\RR^d$.
Let $A \in \RR^{d\times d}$ be the matrix whose columns are $a_1,\cdots,a_d$.
Then $A$ is called a generator matrix of the lattice $L$. We can write $L$ as
$L_A$ and a short notation for $L_A$ is $A\ZZ^d$; that is
 \begin{align*}
         L_A = A\ZZ^d = \left\{A k:\ k\in \ZZ^d \right\}.
 \end{align*}
Throughout this paper,  we shall treat a vector in the Euclidean space as a column
 vector whenever needed. As a result,  $x \cdot y = x^{\tr}y$, where $x^{\tr}$
 denotes the transpose of $x$. The dual lattice $L^{\perp}$ of $L$ is given by
 \begin{align*}
 L^{\perp}:=\ &\left\{ x\in \RR^d: \ x \cdot y \in \ZZ \text{ for all } y \in L\right\}.
 \end{align*}
where  $x\cdot y$ denotes the usual Euclidean inner product of $x$ and $y$.
The generator matrix of $L^\perp$ is $A^{-\tr}$.

A bounded set $\Omega\subset \RR^d$ is said to tile  $\RR^d$ with the lattice $L$ if
\begin{align*}
 \sum_{\alpha\in L} \chi_{\Omega} (x+\alpha) = 1, \quad \text{ for almost all }
     x \in \RR^d,
\end{align*}
where $\chi_{\Omega}$ denotes the characteristic function of $\Omega$, which
we write as $\Omega+L= \RR^d$. Tiling and Fourier analysis are closely related
as demonstrated by the Fuglede theorem. Let $\int_{\Omega} f(x) dx$ denote the integration
of the function $f$ over $\Omega$. Let $\langle \cdot, \cdot \rangle_{\Omega}$
denote the inner product in $L^2({\Omega})$,
\begin{align}
\label{eq:innerproduct}
\langle f,g\rangle_{\Omega} := \frac{1}{|\Omega|} \int_{\Omega} f(x) \overline{g(x)} dx,
\end{align}
where $|\Omega|$ denotes the measure of $\Omega$. The following
fundamental result was proved by Fuglede in \cite{F}.

\begin{thm}  \label{Fuglede}
Let $\Omega\subset\RR^d$ be a bounded
domain and $L$ be a lattice of $\RR^d$. Then $\Omega+L= \RR^d$ if and only
if $\left\{e^{2\pi i\, \alpha \cdot x}: \alpha \in L^{\perp} \right\}$
is an orthonormal basis with respect to the inner product \eqref{eq:innerproduct}.
\end{thm}

The orthonormal property is defined with respect to the normalized Lebesgue
measure on $\Omega$. If $L = L_A$, then the measure of $\Omega$ is equal to
$\sqrt{\det (A^{\tr}A)}$. Furthermore, we can write $\a \in L_A^\perp = A^{-\tr}\ZZ^d$
as $\alpha = A^{-\tr} k$ with $k \in  \ZZ^d$, so that
$\alpha\cdot x = k^{\tr} A^{-1} x$. Hence the orthogonality in the theorem is
\begin{align}
\label{eq:orthonormal}
  \frac{1}{\sqrt{\det (A^{\tr}A)}} \int_{\Omega} \e^{2\pi i\, k^{\tr} A^{-1} x} dx =
      \delta_{k,0}, \quad k\in \ZZ^d.
\end{align}
The set $\Omega$ is called a spectral set (fundamental region) for the lattice $L$.
If $L=L_A$ we also write $\Omega= \Omega_A$.

A function $f\in L^1(\Omega_A)$ can be expanded into a Fourier series
\begin{align*}
 f(x) \sim \sum_{k\in \ZZ^d} c_k \e^{2\pi i\, k^{\tr} A^{-1} x }, \qquad c_k =
  \frac{1}{\sqrt{\det(A^{\tr}A)}}\int_{\Omega} f(x) \e^{-2\pi i\, k^{\tr} A^{-1} x} dx.
\end{align*}
The Fourier transform $\widehat{f}$ of a function defined on $L^1({\RR^d})$ and
its inversion are defined by
\begin{align*}
\widehat{f}(\xi ) = \int_{\RR^d} f(x) \e^{-2\pi i\, \xi \cdot x } dx ,\quad
{f}(x) = \int_{\RR^d} \widehat{f}(\xi) \e^{2\pi i\, \xi \cdot x } d\xi.
\end{align*}
Our first result is the following sampling theorem (see, for example,
\cite{Hi,Ma}).

\begin{prop} \label{pro:interpolation}
Let $\Omega$
 be the spectral set of the lattice $A\ZZ^d$. Assume that $\widehat{f}$ is
supported on $\Omega$ and $\widehat{f}\in L^2(\Omega)$. Then
\begin{align*}
   f(x) = \sum_{k\in \ZZ^d} f(A^{-\tr}k) \Phi_{\Omega} (x-A^{-\tr}k)
\end{align*}
in $L^2(\Omega)$, where
\begin{align*}
\Phi_{\Omega} (x) = \frac{1}{\sqrt{\det(A^{\tr}A)}} \int_{\Omega} \e^{2\pi i \xi \cdot x} d\xi.
\end{align*}
\end{prop}

This theorem is a consequence of the Poisson summation formula. We notice that
\begin{align*}
\Phi_{\Omega} (A^{-\tr}j) = \delta_{0,j}, \quad \text{ for all } j\in \ZZ^d,
\end{align*}
by Theorem \ref{Fuglede}, so that $\Phi_{\Omega}$ can be considered as a cardinal interpolation function.


\subsection{Discrete Fourier analysis and interpolation}
A function $f$ defined on $\RR^d$ is called {\it periodic} with respect to the lattice
$A\ZZ^d$ if
\begin{align*}
          f(x+Ak) = f(x)\qquad \text{ for all } k\in \ZZ^d.
\end{align*}
The spectral set $\Omega$ of the lattice $A\ZZ^d$ is not unique. In order to carry
out the discrete Fourier analysis with respect to the lattice, we shall fix $\Omega$
such that $\Omega$ contains $0$ in its interior and we further require that $\Omega$
tiles $\RR^d$ with $L_A$ without overlapping and without gap. In other words, we
require that
\begin{align}
\label{eq:tiling2}
 \sum_{k\in \ZZ^d}\chi_{\Omega} (x+Ak) = 1, \qquad \text{ for all }  x\in \RR^d.
\end{align}
For example, we can take $\Omega=[-\frac12,\frac12)^d$ for the standard cubic
lattice $\ZZ^d$.

\begin{defn} \label{def:N}
Let $A$ and $B$ be two nonsingular matrices in $\RR^{d\times d}$,
$\Omega_A$ and $\Omega_B$ satisfy \eqref{eq:tiling2}. Assume all entries of
$N:= B^\tr A$ are integers. Define
\begin{align*}
 \Lambda_N    := \left\{k\in \ZZ^d:\ B^{-\tr}k\in \Omega_A \right\}\ \text{  and  }\
 \Lambda_N^{\dag} := \left\{k\in \ZZ^d:\ A^{-\tr}k\in \Omega_B \right\}.
\end{align*}
\end{defn}

Two points $x,y\in \RR^d$ are said to be congruent with respect to the lattice
$A\ZZ^d$, if $x-y\in A\ZZ^d$, and we write $x\equiv y \pmod{A}$. The following
two theorems are the central results for the discrete Fourier transform.

\begin{thm}  \label{th:DFT}
Let $A,\,B$ and $N$ be as in Definition \ref{def:N}. Then
\begin{align*}
\frac{1}{|\det(N)|} \sum_{j\in \Lambda_N} \e^{2\pi i\, k^{\tr} N^{-1} j}
 =\begin{cases} 1, & \text{if }\ k\equiv 0 \pmod{N^{\tr}}, \\
  0, & \text{otherwise}.
   \end{cases}
\end{align*}
\end{thm}

\begin{thm} \label{thm:2.4}
Let $A,\, B$ and $N$ be as in Definition \ref{def:N}. Define the discrete inner product
\begin{align*}
  \langle f,g \rangle_N = \frac{1}{|\det(N)|} \sum_{j\in \Lambda_N} f(B^{-\tr}j)
    \overline{ g(B^{-\tr}j)}
\end{align*}
for $f,\, g \in C(\Omega_A)$, the space of continuous functions on $\Omega_A$. Then
\begin{align}
\label{eq:equiv_inner}
 \langle f,\,g \rangle_N = \langle f,\,g \rangle
\end{align}
for all $f,\,g$ in the finite dimensional subspace
\begin{align*}
  \mathcal{H}_N := \mathrm{span}\left\{\phi_k:\ \phi_k(x) = \e^{2\pi i\, k^{\tr} A^{-1} x},\ k \in \Lambda_N^\dag \right\}.
\end{align*}
\end{thm}

Let $|E|$ denote the cardinality of the set $E$. Then the dimension of
$\mathcal{H}_N$ is $|\Lambda_N^\dag|$.

Let $\mathcal{I}_N f$ denote the Fourier expansion of $f\in C(\Omega_A)$ in
$\mathcal{H}_N$ with respect to the inner product $\langle\cdot, \cdot\rangle_N$.
Then, analogous to the sampling theorem in Proposition \ref{pro:interpolation},
$\CI_Nf$ satisfies the following formula
\begin{align*}
 \mathcal{I}_N f(x) =\sum_{j\in \Lambda_N} f(B^{-\tr}j) \Phi^A_{\Omega_B} (x-B^{-\tr}j), \quad f\in C(\Omega_A),
\end{align*}
where
\begin{align*}
\Phi^A_{\Omega_B} (x) = \frac{1}{|\det(N)|} \sum_{k\in \Lambda_N^\dag}
      \e^{2\pi i\, k^{\tr} A^{-1} x}.
\end{align*}
The following theorem shows that $\CI_N f$ is an interpolation function.

\begin{thm} \label{thm:interpolation}
Let $A$, $B$ and $N$ be as in Definition \ref{def:N}. If
$\Lambda_N^\dag = \Lambda_{N^\tr}$, then $\mathcal{I}_N f$ is
the unique interpolation operator on $N$ in $\mathcal{H}_N$; that is
\begin{align*}
  \mathcal{I}_N f(B^{-\tr}j) = f(B^{-\tr}j), \quad \forall j\in \Lambda_N.
\end{align*}
In particular, $|\Lambda_N | = |\Lambda_N^\dag|$. Furthermore, the fundamental interpolation function
$\Phi^A_{\Omega_B}$ satisfies
\begin{align*}
\Phi^A_{\Omega_B}(x) = \sum_{k\in \ZZ^d} \Phi_{\Omega_B}(x+Aj).
\end{align*}
\end{thm}

The above results have been used to develop a discrete Fourier analysis on a
hexagon in \cite{LSX}. In the following section, we apply it to the rhombic
dodecahedron.


\section{Discrete Fourier analysis on the rhombic dodecahedron}
\setcounter{equation}{0}

In this section we develop a discrete Fourier analysis on the rhombic dodecahedron.
It contains five subsections.


\subsection{Face-centered cubic lattice and Fourier analysis}
We consider the face-centered cubic (fcc) lattice given in Figure 3.1.

\begin{figure}[ht]
\centering
\includegraphics[width=0.6\textwidth]{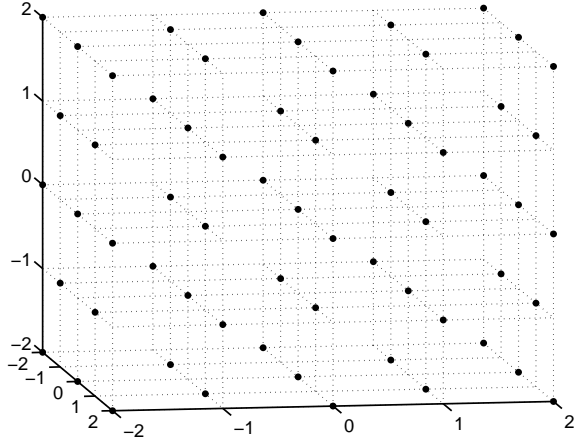}
\caption{Face-centered cubic (fcc) lattice.}
\end{figure}

\noindent
Just like the hexagon lattice, the fcc lattice offers the densest packing of $\RR^3$
with unit balls, which is the so-called Kepler's conjecture and proved recently in
\cite{H}.

The generator matrix $A$ of fcc lattice is given by of
\begin{align*}
   A  = \begin{pmatrix}
               0  &  1 &  1 \\
               1  &  0 &  1 \\
               1  &  1 &  0
              \end{pmatrix}.
\end{align*}
The domain that tiles $\RR^3$ with fcc lattice is the rhombic dodecahedron
(see Figure 3.2). Thus, the spectral set of fcc is
$$
        \Omega = \{x \in \RR^3:  -1<x_j \pm x_i \le 1,  1 \le i<j \le 3 \}.
$$
The strict inequality in the definition of $\Omega$ reflects our requirement that
the tiling of the spectral set has no overlapping.

\begin{figure}[htb]
\centering
\includegraphics[width=0.6\textwidth]{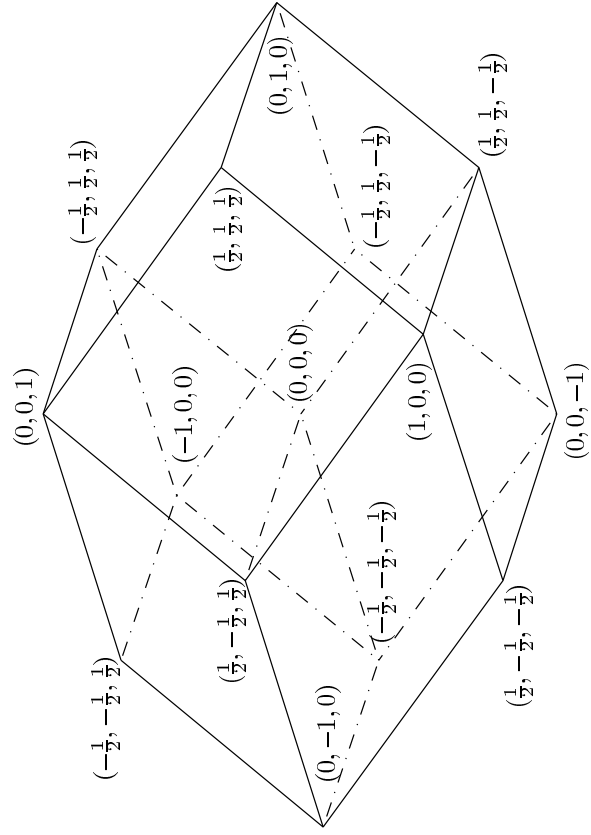}
\caption{Rhombic dodecahedron}
\end{figure}

Motivated by the study of \cite{LSX, Sun, Sun2}, we shall use homogeneous
coordinates $\tb = (t_1,t_2,t_3,t_4)$, where $t_1+t_2+t_3+t_4=0$, in $\RR^4$
for our analysis on the rhombic dodecahedron in $\RR^3$. The advantage is
that our formulas become more symmetric and the symmetry becomes more
transparent under the homogeneous coordinates.  Throughout the rest of this
paper, we adopt the convention of using bold letters,
such as $\tb$, to denote the points in the space
\begin{align*}
 \RR_H^4 := \left\{ \tb = (t_1,t_2,t_3,t_4)\in \RR^4:
  t_1+t_2+t_3+t_4=0 \right\}.
\end{align*}
In other words, the bold letters such as $\tb$ and $\kb$ will always mean
homogeneous coordinates.  The transformation between $x \in \RR^3$ and
$\tb \in \RR_H^4$ is defined by
\begin{equation}\label{coordinate1}
  x = A \begin{pmatrix} t_1 \\t_2\\t_3 \end{pmatrix}  \quad \Longleftrightarrow \quad
       \begin{cases} x_1= t_2 + t_3 \\
            x_2 =   t_1+ t_3 \\
            x_3=    t_2 + t_1
          \end{cases}
\end{equation}
and $t_4 = -t_1-t_2-t_3$. Let us denote by $H$ and $U$ the matrices
\begin{equation*}
   H =  \begin{pmatrix} 1 & 0 & 0\\ 0 & 1 & 0\\ 0 & 0 & 1\\ -1 & -1 & -1
          \end{pmatrix}
\quad\hbox{and}\quad
  U = \frac12 \begin{pmatrix}
   -1  &  1 &  1 \\ 1  & -1 &  1 \\ 1  &  1 & -1 \\ -1  & -1 & -1
\end{pmatrix},
\end{equation*}
respectively. The columns of the matrix $U$ are orthonormal and $U^\tr U = I$.
We then have $A = U^\tr H$ and the inverse transform of \eqref{coordinate1} is
 \begin{equation} \label{coordinate}
   \tb = U x
    \quad \Longleftrightarrow \quad
       \begin{cases} t_1=  \tfrac{1}{2}(-x_1+x_2+x_3) \\
            t_2 =   \tfrac{1}{2}(x_1-x_2+x_3) \\
            t_3=   \tfrac{1}{2} (x_1+x_2-x_3) \\
            t_4 =  \tfrac{1}{2}(-x_1-x_2-x_3).
       \end{cases}
\end{equation}

In the homogenous coordinates, the spectral set is $\Omega_H: =
\{\tb=U x: x \in \Omega\}$ which, upon using \eqref{coordinate}, results to
\begin{equation}\label{Omega}
   \Omega_H = \left\{\tb \in \RR_H^4:
             -1< t_i-t_j  \leq  1,  1\leq i < j \leq 4\right\}.
\end{equation}
Figure 3.3 shows again the rhombic dodecahedron with vertices labeled in
the homogeneous coordinates.

\begin{figure}[ht]
\centering
\includegraphics[width=0.6\textwidth]{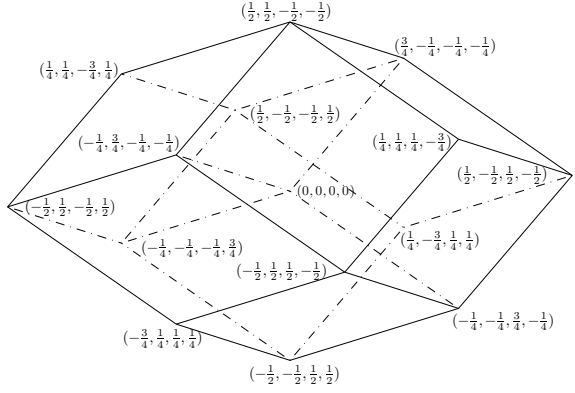}
\caption{Rhombic dodecahedron labeled in homogeneous coordinates.}
\end{figure}

Under the change of variables \eqref{coordinate}, the integral over $\Omega_H$ is
given by
\begin{equation} \label{integral}
   \frac{1}{|\Omega_H|}\int_{\Omega_H} f(\tb)d\tb =  \frac{1}{|\Omega |}
      \int_{\Omega} f(x) dx.
\end{equation}

By Fuglede theorem, $\left\{e^{2\pi i\, k^{\tr} A^{-1} x}: k \in \ZZ^3  \right\}$
forms an orthonormal basis in $L^2(\Omega)$. We would like to reformulate
the exponential functions $e^{2\pi i\, k^{\tr} A^{-1} x}$ so that they are
indexed by homogeneous coordinates. For this purpose, we denote by
\begin{equation*}
\ZZ^4_H : = \ZZ^4 \cap \RR^4_H = \left\{ \kb \in \ZZ^4:
        k_1+k_2+k_3+k_4=0  \right\}
\end{equation*}
the set of integers in homogeneous coordinates and introduce the notation
\begin{equation} \label{HH}
   \HH : =  \left\{ \kb \in\ZZ^4_H:
   k_1\equiv k_2\equiv k_3 \equiv k_4\!\!\!\pmod{4}\right\}.
\end{equation}
The definition of the matrix $H$ shows that if $k \in \ZZ^3$ then $H k \in \ZZ_H^4$.
For $k \in \ZZ^3$, set $\jb = 4 H (A^\tr A)^{-1} k \in \ZZ_H^4$. A quick computation
reveals that $\jb \in \HH$. Moreover, given $\jb \in \HH$, it follows from the fact
$A= U^\tr H$ that $k = \frac{1}{4} A^\tr U^\tr \jb$, which is easily seen to be in $\ZZ^3$. Furthermore, we have
$$
   k^\tr A^{-1} x = \tfrac{1}{4} \jb^\tr U A A^{-1}x  = \tfrac{1}{4} \jb^\tr U x  =
        \tfrac{1}{4} \jb \cdot \tb.
$$
Consequently, we can index the exponentials by $\jb \in \HH$ and the exponent
$2 \pi i k^\tr A^{-1}x$ becomes $\frac{\pi i}{2} \jb \cdot \tb$. Let us introduce the
notation
\begin{align}\label{phi}
   \phi_{\jb}(\tb) : =  e^{\frac{\pi i}{2} \jb \cdot \tb},    \qquad   \jb \in \HH.
\end{align}
Then, using \eqref{coordinate} and recalling \eqref{eq:orthonormal},
the Fuglede Theorem becomes the following:

\begin{prop}  \label{prop:ortho-H}
For $\kb,\jb \in \HH$,
\begin{align*}
    \langle \phi_{\kb},  \phi_{\jb}\rangle
    = \frac12 \int_{\Omega_H} \phi_{\kb}(\tb) \overline{\phi_{\jb}(\tb)} d\tb
     = \delta_{\kb, \jb}. 
\end{align*}
Furthermore, $\{\phi_\jb: \jb \in \HH\}$ is an orthonormal basis of
$L^2(\Omega_H)$.
\end{prop}

Given $f$ defined on $\Omega$, the mapping \eqref{coordinate} shows that
$f(x) = f(U^\tr \tb) = g(\tb)$ is the function in homogeneous coordinates. Since
$A = U^\tr H$, a function $f$ being periodic with respect to the lattice $A\ZZ^d$
becomes, in homogeneous coordinates, the following definition:

\begin{defn} \label{H-periodic}
A function $f$ is $H$-periodic if it is periodic with respect to the fcc lattice;
that is, $f(\tb) = f(\tb + H k)$ for $x\in \Omega_H$ and $k \in \ZZ^3$.
\end{defn}

Using the explicit form of the matrix $H$, it is easy to see that the following holds:

\begin{lem} \label{lm:periodicity}
A function $f(\tb)$ is $H$-periodic if and only if
\begin{align*}
 f(\tb) = f(\sb), \quad \tb- \sb \in \ZZ_H^4,
\end{align*}
or equivalently
\begin{align*}
 f(\tb+k \eb_{i,j}) = f(\tb), \qquad k \in \ZZ,  \quad 1\leq i<j\leq 4,
\end{align*}
where $\eb_{i,j}: = e_i - e_j$ and $\{e_1, e_1,e_3,e_4\}$ is the standard basis
of $\RR^4$.
\end{lem}

Evidently, the functions $\phi_\jb(\tb)$ in \eqref{phi} are $H$-periodic. Furthermore,
Proposition \ref{prop:ortho-H} shows that an $H$-periodic function $f$ can be
expanded into a Fourier series
\begin{equation}\label{H-Fourier}
   f \sim \sum_{\kb \in \HH } \wh f_\kb \phi_\kb(\tb), \quad\hbox{where} \quad
       \wh f_\kb := \frac12 \int_{\Omega_H} f(\tb) \phi_{-\kb}(\tb) d \tb.
\end{equation}


\subsection{Boundary of the rhombic dodecahedron}
In order to carry out the discrete Fourier analysis on the rhombic dodecahedron,
we need to have a detailed knowledge of the boundary of the polyhedral.

We use the standard set theory notations $\partial \Omega$,  $\Omega^\circ$
and $\overline{\Omega}$ to denote the boundary, the interior and the closure
of $\Omega$, respectively. Clearly $\overline{\Omega} = \Omega^\circ \cup
\partial \Omega$. A rhombic dodecahedron has 12 faces,  24 edges and 14
vertices.  Since we will consider points on the boundary, we need to
distinguish a face with its edges and without its edges, and an edge
with its end points and without its end points. In the following,
when we say a face or an edge, we mean the {\it open} set, that is,
without its edges or end points, respectively.

We shall work with homogeneous coordinates. To describe the boundary of
$\Omega_H$ we set $\NN_4: = \{1,2,3,4\}$. For $i,j \in \NN_4$ and $i\ne j$, define
$$
         F_{i,j}  = \{ \tb \in \overline{ \Omega}_H: t_i - t_j =1\}.
$$
There are a total $2 \binom{4}{2}  = 12$ distinct $F_{i,j}$ and it easy to see that
each $F_{i,j}$ stands for one face, with its edges, of the rhombic dodecahedron.

For nonempty subsets $I, J$ of $\NN_4$, define
\begin{align*}
    \Omega_{I,J} :=  \bigcap_{i\in I, j\in J} F_{i,j} =
        \left\{ \tb \in \overline{ \Omega}_H:\  t_j = t_i-1, \text{ for all } i\in I,\
  j\in J\right\}.
\end{align*}

\begin{lem}  \label{lm:boundary}
Let $I, J, I_i, J_i$ be nonempty subsets of $\NN_4$. Then
\begin{enumerate}
\item[(i)] $\Omega_{I,J} = \emptyset$ if and only if $I\cap J \neq \emptyset$.
\item[(ii)]  $ \Omega_{I_1,J_1} \cap \Omega_{I_2,J_2} =\Omega_{I, J}$  if
  $I_1 \cup I_2=I$ and $J_1\cup J_2=J$.
\end{enumerate}
\end{lem}

\begin{proof}
It is obvious that $\Omega_{I,J} \neq  \emptyset$ if $I\cap J = \emptyset$. On the
other hand, if $I\cap J \neq \emptyset$ and $i= j\in I\cap J$,  then $t_i-t_j=0 \neq 1$,  which shows that $\Omega_{I,J}=\emptyset$. This proves (i).

If either $I_1 \cap J_1 \neq \emptyset$ or $I_2 \cap J_2 \neq \emptyset$, then
$\Omega_{I_1,J_1}\cap \Omega_{I_2,J_2} =\Omega_{I, J} = \emptyset$ by (i).
If $I_1 \cap J_1 =  I_2 \cap J_2 =\emptyset$ and $i_\nu \in I_\nu,
 j_\nu \in J_\nu$ for $\nu=1,2$, then
 \begin{align*}
        t_{i_1} - t_{j_2}  =   t_{i_1} - t_{i_2}+1 \leq 1 \quad\hbox{and}\quad
                  -1 \leq t_{j_1} - t_{i_2}  =   t_{i_1} - t_{i_2}-1,
 \end{align*}
which implies $t_{i_1} -t_{i_2} =0$ so that $t_{j_1} =t_{j_2} = t_{i_1}-1 =t_{i_2}-1$
and proves $\Omega_{I_1,J_1}\cap \Omega_{I_2,J_2} =\Omega_{I, J}$.
 \end{proof}

Edges are intersections of faces and vertices are intersections of edges.
The Lemma \ref{lm:boundary} gives us information about the intersections.
To make clear the structure of the boundary $\partial \Omega_H$, we introduce
the notation
\begin{align*}
  &\CK = \left\{(I,J):  I, J \subset \NN_4; \   I\cap J = \emptyset \right\},\\
  &\CK_0 = \left\{ (I,J)\in \CK:\ i<j, \,\, \hbox{for all}\,\, (i,j) \in (I, J) \right\}.
\end{align*}

\begin{defn}
For $(I,J)\in \CK$, the boundary element $B_{I,J}$ of the dodecahedron,
\begin{align*}
  B_{I,J}: = \left\{\tb \in \Omega_{I,J}: \  \tb \not \in \Omega_{I_1,J_1}
\text{ for all }
  (I_1,J_1)\in \CK \text{ with } |I|+|J| < |I_1| + |J_1| \right\},
\end{align*}
is called a face if $|I|+|J| =2$, an edge if $|I|+|J|=3$ and a vertex if $|I|+|J|=4$.
\end{defn}

For the faces and the edges, the boundary elements represent the interiors.
In fact, it is easy to see that $B_{\{i\},\{j\}} = F_{i,j}^\circ$ and, for example,
$B_{\{i\}, \{j,k\}} = (F_{i,j} \cap F_{i,k})^\circ$ for distinct integers $i, j ,k \in \NN_4$.

Furthermore,  for $0 < i,j < i+j \le 4$, we define
\begin{align} \label{CK}
\begin{split}
   \CK^{i,j}: =&\, \left\{(I,J) \in \CK:\  |I| = i,\ |J| =j  \right\}, \quad
       B^{i,j}: = \bigcup_{(I, J) \in \CK^{i,j}}  B_{I,J} \\
   \CK^{i,j}_0: = &\, \left\{(I,J) \in \CK_0:\  |I| = i,\ |J| =j  \right\}, \quad
       B^{i,j}_0: =\bigcup_{(I, J) \in \CK_0^{i,j}}  B_{I,J}.
\end{split}
\end{align}
Note that $B^{i,j}$ is the union of boundary points in those $B_{I,J}$
for which $|I|=i$ and $|J|=j$.

\begin{prop} \label{prop:BIJ} 
Let $(I,J) \in \CK$ and $(I_1,J_1)\in \CK$.
\begin{enumerate}
\item[(i)]  $ B_{I,J} \cap B_{I_1,J_1} = \emptyset$,  if $I \neq I_1$ and $J\neq J_1$.
\item[(ii)] $\overline{\Omega}_H\setminus \Omega_H^{\circ}
 =  \bigcup_{(I,J)\in \CK}  B_{I,J}  = \bigcup_{0<i,j<i+j\le 4} B^{i,j}$.
 \item[(iii)] $\Omega_H\setminus \Omega_H^{\circ}
 =  \bigcup_{(I,J)\in \CK_0}  B_{I,J}  = \bigcup_{0<i,j<i+j\le 4} B_0^{i,j}$.
\end{enumerate}
\end{prop}

\begin{proof}
If $B_{I,J}\cap B_{I_1,J_1}\neq \emptyset$ then $|I|+|J| = |I_1|+|J_1|$.
Moreover, if $\tb \in B_{I,J}\cap B_{I_1,J_1}$  then $\tb\in \Omega_{I,J}\cap
\Omega_{I_1,J_1} = \Omega_{I\cup I_1,J\cup J_1}$ by Lemma \ref{lm:boundary},
which implies that $|I|+|J|\geq  |I\cup I_1| +|J\cup J_1|$. Thus we must have
$I=I_1=I\cup I_1$ and $J=J_1=J\cup J_1$, which contradicts the assumption
and proves (i).

To prove (ii), we define for $\tb \in \partial \Omega_H$,
$I = \left\{ i\in \NN_4:  \exists j \in \NN_4 \text{ such that } t_i-t_j = 1 \right\}$ and
$J = \left\{ j\in \NN_4:  \exists i \in \NN_4 \text{ such that } t_i-t_j = 1 \right\}$.
Clearly $(I,J)\in \CK$ and $\tb \in B_{I,J}$, which proves the first equal sign of
(ii). The second equal sign follows from the definition of $B^{i,j}$.
Since
\begin{align*}
\Omega_H\setminus\Omega_H^{\circ} =\left\{ \tb \in \overline{\Omega}_H\setminus
\Omega_H^\circ: t_i-t_j  > -1, \,\,   \forall i<j \right\},
\end{align*}
the part (iii) follows immediately from (ii).
\end{proof}

The above proposition provides a decomposition of the boundary into
non-overlapping boundary elements. To make each boundary element
explicit, we use symmetry.  Let $\mathcal{G} =S_4$ be the
permutation group of four elements. For $\tb \in \RR^4_H$ and
$\sigma \in \CG$,  the action of $\sigma$ on $\tb$ is denoted by
$\tb \sigma$, which means the permutation of the elements of $\tb$
by $\sigma$. A moment of reflection shows that, for $(I,J) \subset
\CK$,
$$
     B^{|I|,|J|} = \bigcup_{\sigma \in \G} B_{I,J}\sigma
     : = \left\{ \tb \sigma: \tb \in B_{I,J},  \sigma\in \mathcal{G} \right\}.
$$

Later in the section we will need to consider points on the boundary elements
that are congruent module $H$.  For $(I,J) \subset \CK$ we further define
\begin{align*}
   [B_{I,J}]:= \left\{ B_{I,J} + \kb : \kb \in \ZZ_H^4\right\}\cap
       \overline{\Omega}_H
     = \left\{ \tb+\kb\in \overline{\Omega}_H:  \tb \in B_{I,J},  \kb \in \ZZ_H^4 \right\}.
\end{align*}
Since $[B_{I,J}]$ is a subset of $\overline{\Omega}_H$ and $B_{I,J}$ is
a boundary element, we see that $[B_{I,J}]$ consists of exactly those
boundary elements that can be obtained from $B_{I,J}$ by congruent modulus
$H$, as confirmed by the following lemma.

\begin{lem} \label{lem:s=equiv=t} 
If $\tb,\sb \in \Omega_H$  and $\sb \equiv \tb \pmod{H}$, then $\tb=\sb$.
\end{lem}

\begin{proof}
By Lemma \ref{lm:periodicity}, if $\tb, \sb \in \Omega_H$ and $\tb \equiv \sb
\pmod{H}$, then $\sb - \tb \in \ZZ_H^4$ and, set $\kb := \sb-\tb$,
$-1\leq  k_i  -k_j \leq 1$  for all $i,j\in \NN_4$. The last condition means that
either $k_i \in \left\{0,1 \right\}$ for all $ i \in \NN_4$ or $ k_i \in \left\{0,-1\right\}$
for all $i \in \NN_4$. The homogenous condition $k_1+k_2+k_3+k_4=0$ then
shows that $k_1=k_2 =k_3=k_4=0$ or $\sb=\tb$.
\end{proof}

As an example, we have
\begin{align} \label{B_{1,{2,3}}}
\begin{split}
& B_{\{1\},\{2,3\}} = \left \{(t, t-1, t-1, 2-3t):  \tfrac12 < t <\tfrac34  \right \}, \\
& B_{\{1,2\},\{3\}}= \left\{ (1-t, 1-t, -t, 3t-2):  \   \tfrac12 < t <\tfrac34  \right\},
\end{split}
\end{align}
and from the explicit description of $B_{\{1\},\{2,3\}}$ we deduce
\begin{align} \label{[B{1,23}]}
  [B_{\{1\},\{2,3\}}] & = B_{\{1\},\{2,3\}} \cup \left (B_{\{1\},\{2,3\}} +(-1,1,0,0)\right)
     \cup \left (B_{\{1\},\{2,3\}} +(-1,0,1,0)\right)  \notag \\
       & = B_{\{1\},\{2,3\}} \cup B_{\{2\},\{1,3\}} \cup B_{\{3\},\{1,2\}}.
\end{align}
Others can be deduced similarly. The last equation indicates that $[B_{I,J}]$ is
a union of $B_{I',J'}$, which we make precise below.

Let $\sigma_{ij}$ denote the element in $\mathcal{G}$ that interchanges $i$
and $j$; then $\tb \sigma_{ij} = \tb - (t_i-t_j) \eb_{i,j}$. For a nonempty set
$I\subset \mathbb{N}_4$, define $\mathcal{G}_{I} := \left\{\sigma_{ij}:  i,j\in I\right\}$,
where we take $\sigma_{ij} = \sigma_{ji}$ and take $\sigma_{jj}$ as the identity
element. It is easy to verify that $\mathcal{G}_{I} $ forms a subgroup of
$\mathcal{G}= \mathcal{S}_4$ of order $|I|$.

\begin{lem}  \label{congruence} 
Let $(I,J)\in \mathcal{K}$. Then
\begin{align} \label{eq:[B]}
   [B_{I,J}] = \bigcup_{\sigma\in \mathcal{G}_{I\cup J}} B_{I,J}\sigma.
\end{align}
\end{lem}

\begin{proof}
For any $i,j\in I\cup J$, the definition of $B_{I,J}$ shows that
$\tb\sigma_{ij}-\tb = (t_j-t_i)\eb_{ij} \in \ZZ_H^4$ for all $\tb\in B_{I,J}$.
It then follows from $\tb\sigma_{ij}\in \overline{\Omega}_H$ that
$B_{I,J}\sigma_{ij} \subseteq [B_{I,J}]$. Consequently,
$\bigcup_{\sigma\in  \mathcal{G}_{I\cup J}} B_{I,J}\sigma \subseteq [B_{I,J} ]$.

On the other hand, for any $\sb \in [B_{I,J}]$ there exists $\tb\in B_{I,J}$ such
that $\sb-\tb\in \ZZ_H^4$. It follows from Proposition \ref{prop:BIJ} and
Lemma \ref{lem:s=equiv=t} that $\sb\in B_{I_1,J_1}$ for a pair $(I_1,J_1)\in \CK$.
By the definitions of $B_{I,J}$ and $\overline{\Omega}_H$, there exist $t, s \in \RR$
such that
\begin{align*}
 &  -\tfrac{3}{4} \leq  t-1=t_j < t_l  < t_i=t \leq  \tfrac{3}{4}, \qquad  i\in I,   \  \
     j\in J, \  l\not \in I\cup J,\\
  & -\tfrac{3}{4} \leq s-1=s_j < s_l <s_i=s \leq \tfrac{3}{4}, \qquad  i\in I_1, \  \
       j\in J_1, \  l\not \in I_1\cup J_1.
\end{align*}
Since $\sb - \tb \in \ZZ_H^4$, the above inequalities imply that
$s_i -t_i \in \{-1, 0,1\}$ for all $i \in \NN_4$. We claim that $t
=s$.  Assume otherwise, say $s > t$. For $i \in I_1$, $s_i-t_i \ge
s-t>0$ so that $s_i-t_i=1$; while for $i \not\in I_1$, $s_i-t_i \ge
s-1-t>-1$ so that $s_i-t_i\in \left\{ 1,0\right\}$. It then follows
that $\sum_{i\in \NN_4}(s_i-t_i)>0$, which poses a contradiction to
the homogeneity of $\sb-\tb$.  Hence we must have $s=t$. With $s
=t$, it is then easy to see that $s_i-s=s_i-t_i\in \left\{0, -1
\right\}$ for $i\in I$,  $s_j-s+1=s_j-t_j \in  \left\{0, 1 \right\}$
for $j\in J$, and $s-1<s_l=t_l<s$ for $l\not \in I\cup J$. This
shows that $I\cup J = I_1\cup J_1$ and $\sum_{i \in I\cup J} t_i =
\sum_{i \in I_1\cup J_1} s_i$. Meanwhile, we note that $\sum_{i \in
I\cup J} t_i = t (|I|+|J|) -|J|$ and $\sum_{i \in I_1\cup J_1} s_i =
s (|I_1|+|J_1|) -|J_1|$. It then follows that  $|J|=|J_1|$ and
$|I|=|I_1|$.
%
%
Consequently, $\sb=\tb \sigma$ for a $\sigma\in \mathcal{G}_{I\cup
J}$ and $[B_{I,J}]\subseteq \bigcup_{\sigma\in \mathcal{G}_{I\cup
J}} B_{I,J}\sigma$. This completes the proof of the lemma.
\end{proof}

Since $\CK^{i,j}$ can be obtained from $\CK_0^{i,j}$ from the action of $\CG$,
it follows that
\begin{equation} \label{eq:Bij=[B]}
  B^{i,j} = \bigcup_{(I,J) \in \CK_0^{i,j}} [B_{I,J}]= \bigcup_{B \in B_0^{i,j}} [B],
        \qquad  0 < i,j < i+j \le 4.
\end{equation}
We also note that $[B_{I,J}]\cap [B_{I_1,J_1}] = \emptyset$ if $(I,J) \neq (I_1,J_1)$
for $(I,J)\in \mathcal{K}_0$ and $(I_1,J_1)\in \mathcal{K}_0$, which shows that
\eqref{eq:Bij=[B]} is a non-overlapping partition.

If $|I|+|J| =3$ or $B_{I,J}$ is an edge, then we have
\begin{align} \label{eq:B12}
\begin{split}
 &B^{1,2}   = [B_{\{1\},\{2,3\}}] \cup  [B_{\{1\},\{2,4\}}]\cup [B_{\{1\},\{3,4\}}]\cup
   [B_{\{2\},\{3,4\}}],\\
& B^{2,1} =  [B_{\{1,2\},\{3\}}]\cup  [B_{\{1,2\},\{4\}}]\cup [B_{\{1,3\},\{4\}}]
  \cup [B_{\{2,3\},\{4\}}],
\end{split}
\end{align}
where, recall that $B_{\{1\},\{2,3\}}$ and $B_{\{1,2\},\{3\}}$ are given in
\eqref{B_{1,{2,3}}},
\begin{align} \label{eq:Bij}
\begin{split}
& B_{\{1\},\{2,4\}} = B_{\{1\},\{2,3\}} \sigma_{34}, \qquad \quad\,
B_{\{1,2\},\{4\}} = B_{\{1,2\},\{3\}} \sigma_{34}, \\
& B_{\{1\},\{3,4\}} =B_{\{1\},\{2,3\}} \sigma_{24},
\quad B_{\{1,3\},\{4\}}= B_{\{1,2\},\{3\}} \sigma_{23}\sigma_{34}, \\
& B_{\{2\},\{3,4\}} = B_{\{1,2\},\{3\}} \sigma_{12} \sigma_{24}, \quad \,\,
B_{\{2,3\},\{4\}} = B_{\{1,2\},\{3\}} \sigma_{13}\sigma_{34}.
\end{split}
\end{align}
If $|I|+|J|=4$, then
\begin{align} \label{eq:B13}
\begin{split}
& B^{1,3} = \left[\{(\tfrac{1}{4}, \tfrac{1}{4}, \tfrac{1}{4},
 -\tfrac{3}{4})\}\right], \quad B^{2,2} = \left[\{(\tfrac{1}{2}, \tfrac{1}{2}, -\tfrac{1}{2},
 -\tfrac{1}{2}\})\right]\\
& B^{3,1} = \left[\{(\tfrac{3}{4}, -\tfrac{1}{4}, -\tfrac{1}{4},  -\tfrac{1}{4})\}\right].
\end{split}
\end{align}


\subsection{Dodecahedral Fourier partial sum}
In order to apply the general result on discrete Fourier analysis in the
previous section to fcc lattice, we choose $A=A$ and $B= nA$ with $n$
being a positive integer.  Then the matrix
\begin{align*}
  N = B^{\tr}A = \begin{pmatrix} 2n & n & n\\ n & 2n & n\\ n & n & 2n
         \end{pmatrix}
\end{align*}
has integer entries. Note that $N$ is now a symmetric matrix so that $\Lambda_N
 = \Lambda_{N^{\tr}}$, and it is easy to see that $\Lambda_N^\dag = \Lambda_N$.
Recall the definition of $\HH$ in \eqref{HH}.  Using again
$\jb = 4H (A^\tr A)^{-1}k \in \ZZ_H^4$, it is easy to see that $k \in \Lambda _N$
becomes $\jb \in \HH_n$, where
$$
\HH_n := \left\{ \kb \in \HH: \tfrac{\kb}{4n}\in \Omega_H\right\}
= \left\{\kb \in \HH:   -4n < k_i -k_j \leq  4n,\,  1\leq i < j \leq 4 \right\}.
$$
The finite dimensional space $\mathcal{H}_N$ of exponentials in
Theorem \ref{thm:2.4} becomes
$$
\mathcal{H}_{n} :=\mathrm{span} \left\{\phi_{\kb}:\  \kb \in \HH_n\right\}
  \qquad \text{with}\quad
\mathrm{dim} \, \mathcal{H}_{n} = \det (N) = 4n^3.
$$

Note that the points in $\HH_n$  are not symmetric under $\CG$, since points
on half of the boundary are not included. For reasons of symmetry, we further
define
$$
      \HH^*_n :=\left\{ \kb \in \HH: \tfrac{\kb}{4n}\in \overline{\Omega}_H \right\}
 = \left\{\kb \in \HH: -4n \leq k_i-k_j\leq 4n,\, 1\leq i< j \leq 4 \right \}.
$$

\begin{figure}[htb]
\centering
\includegraphics[width=0.6\textwidth]{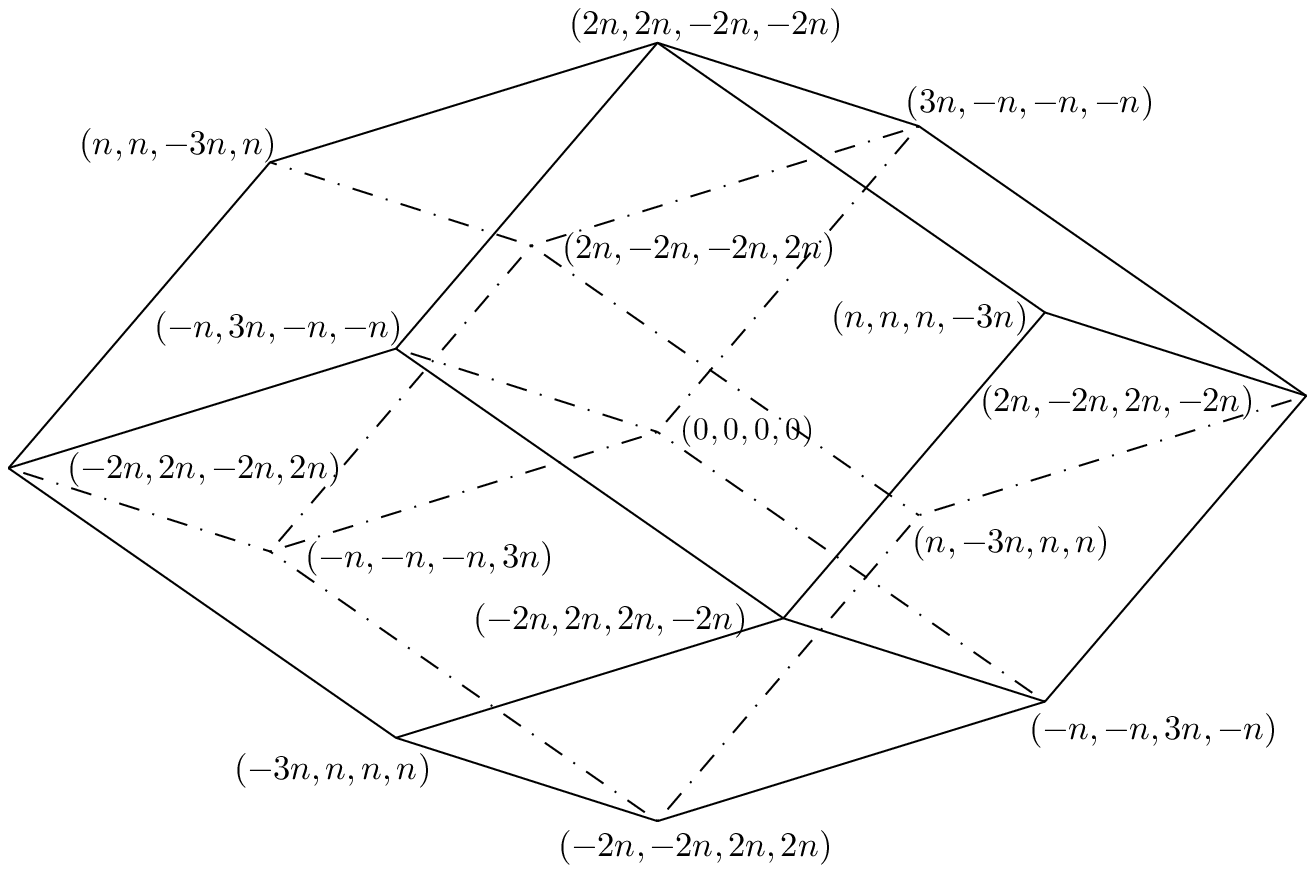}
\caption{$\HH_n^*=\left\{\tb\in \HH:\  -4n \leq  t_i-t_j \leq 4n\right\}$.}
\end{figure}

For the Fourier expansion \eqref{H-Fourier} of an $H$-periodic function, we
define its dodecahedral partial sum as
\begin{align} \label{PartialSum}
S_n f (\tb) := \sum_{\kb\in \HH_n^*} \langle f, \phi_{\kb}  \rangle
 \phi_{\kb}(\tb)
= \frac{1}{2}\int_{\Omega_H} f(\sb) D_n^H(\tb-\sb) d\sb,
\end{align}
where $D_n^H$ is the Dirichlet kernel for the dodecahedral partial sum
\begin{align} \label{def:Dn^H}
D_n^H(\tb) := \sum_{\kb\in \HH_n^*}  e^{\frac{\pi i}{2}\, \kb\cdot \tb}.
\end{align}

Our immediate goal is to find a compact formula for the Dirichlet kernel $D_n^H$.
We start with an observation that the index set $\HH_n^*$ can be partitioned into
four congruent parts, each within a parallelepiped, as shown in Figures 3.5-3.8.

\begin{figure}[htb]
\hfill\begin{minipage}[b]{0.48\textwidth}
\centering
\includegraphics[width=0.9\textwidth]{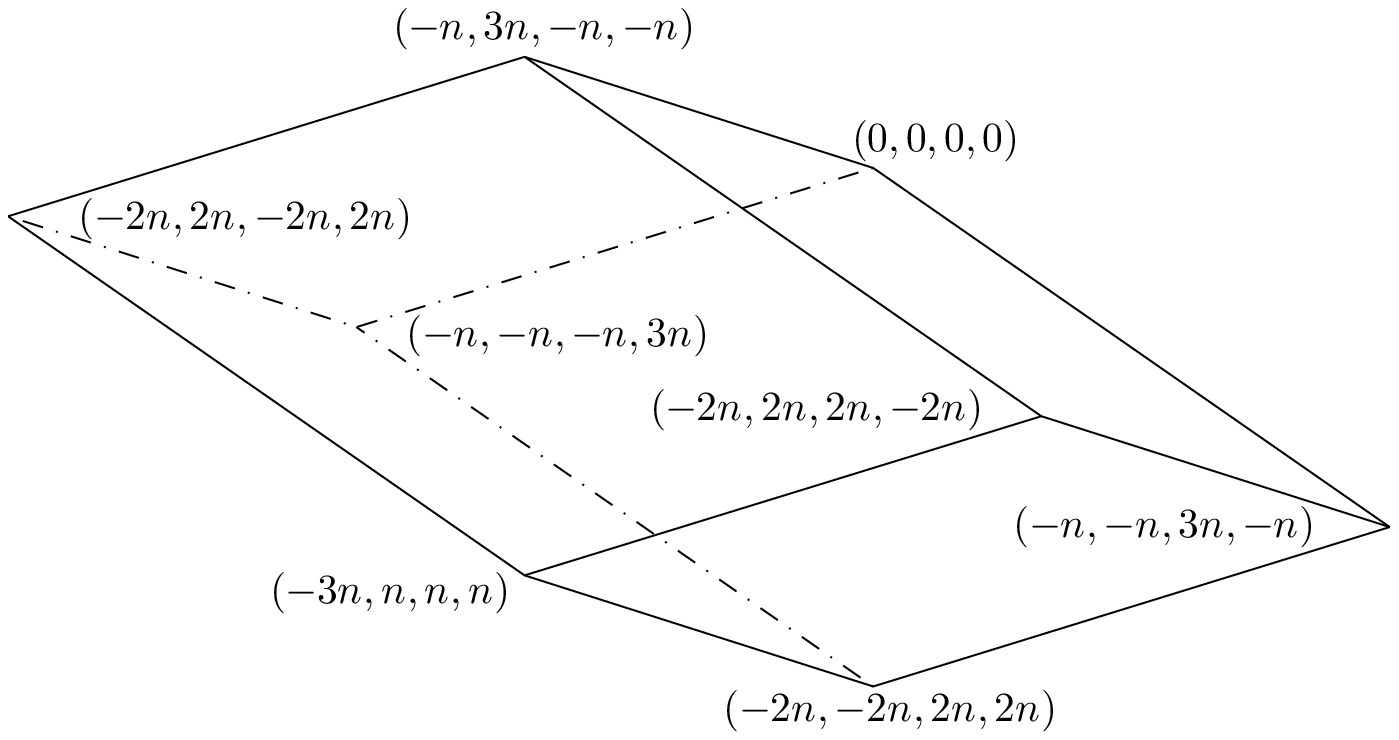}
\caption{$\HH^{(1)}_n$}
\end{minipage}\hfill\begin{minipage}[b]{0.48\textwidth}
\centering
\includegraphics[width=0.75\textwidth]{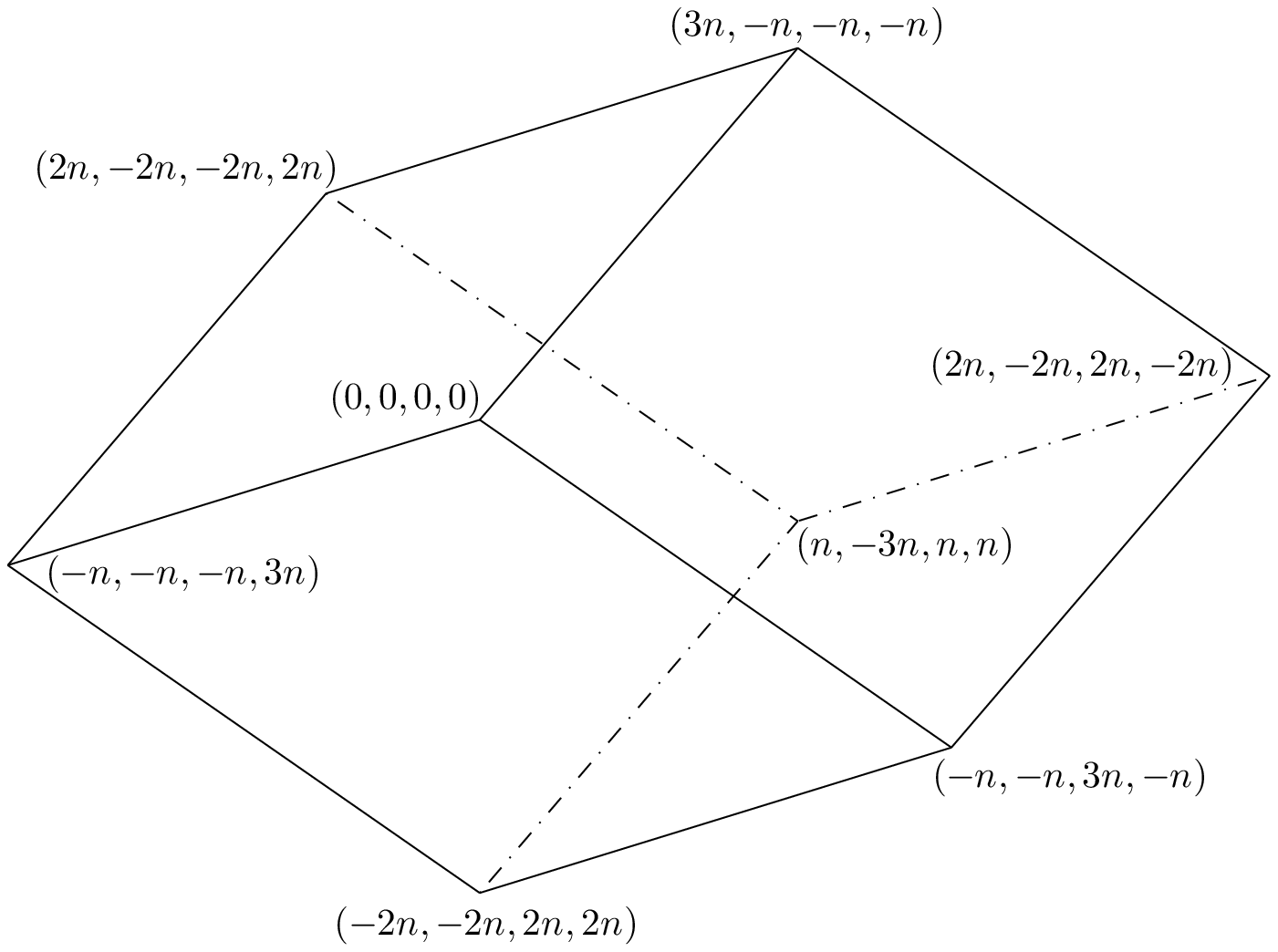}
\caption{$\HH^{(2)}_n$}
\end{minipage}\hspace*{\fill}
\end{figure}

\begin{figure}[htb]
\hfill\begin{minipage}[b]{0.48\textwidth}
\centering
\includegraphics[width=0.9\textwidth]{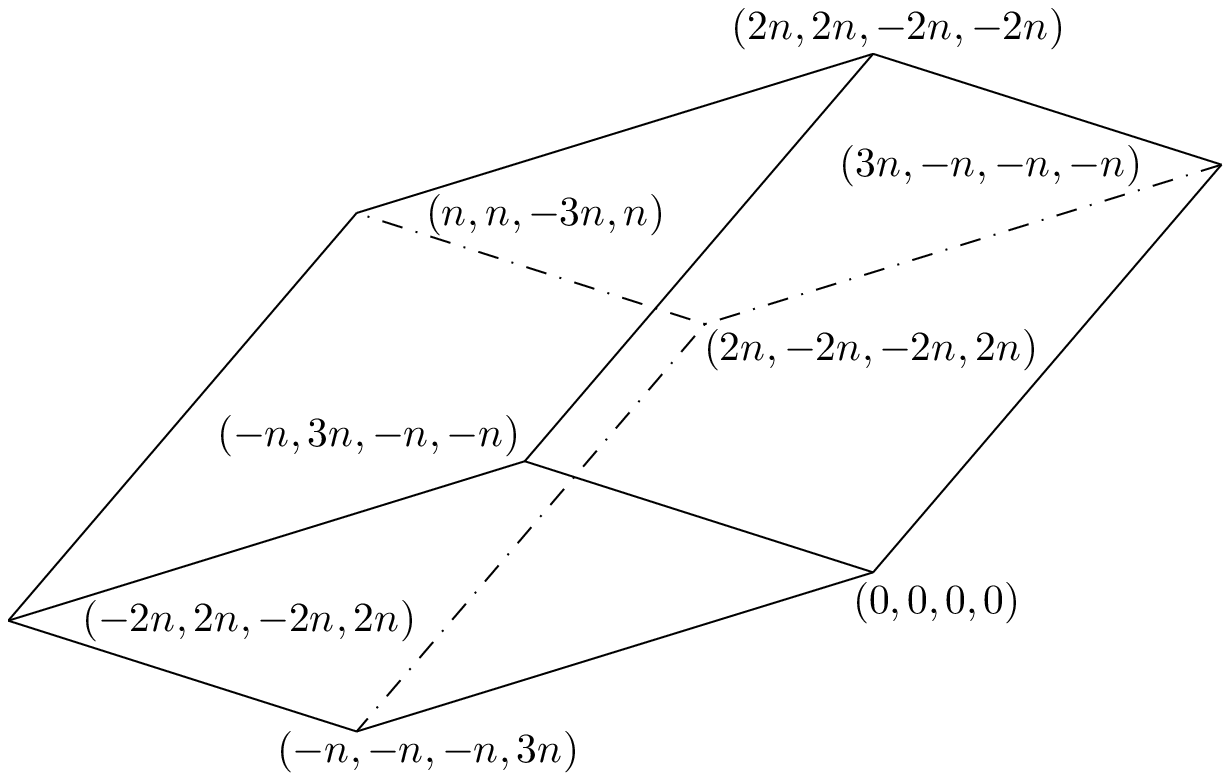}
\caption{$\HH^{(3)}_n$}
\end{minipage}\hfill\begin{minipage}[b]{0.48\textwidth}
\centering
\includegraphics[width=0.75\textwidth]{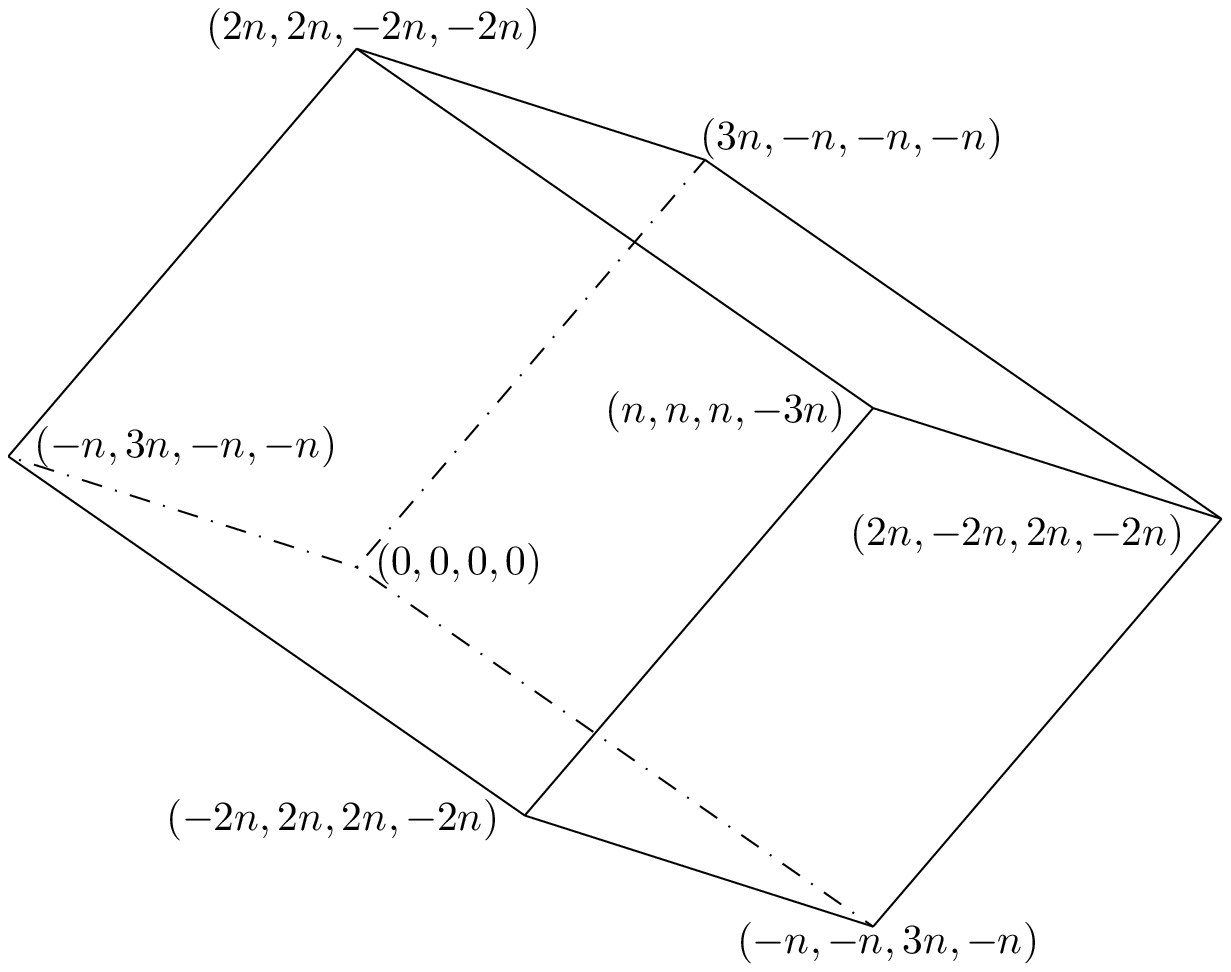}
\caption{$\HH^{(4)}_n$}
\end{minipage}\hspace*{\fill}
\end{figure}

\begin{lem} 
Define $\HH_n^{(k)}: = \left\{ \jb \in \HH:   0\leq j_l -j_k \leq 4 n, \,  l \in \NN_4\right\}$
for $k \in \NN_4$ and
\begin{align*}
\HH_n^{J}:= \left\{\kb \in \HH: k_i = k_j, \, \forall i,j \in J;
\text{ and } \, 0\leq  k_i - k_j \leq 4n,\, \forall j\in J, \ \forall i\in \NN_4\setminus J\right\}
\end{align*}
for $\emptyset\subset J\subseteq  \NN_4$. Then
\begin{align*}
\HH_n^* =  \bigcup_{j\in \NN_4} \HH_n^{(j)} \qquad \hbox{and}\qquad
\HH_n^{J} = \bigcap_{j \in J} \HH_n^{(j)}.
\end{align*}
\end{lem}

\begin{proof}
For $\kb \in \HH_n^*$, let $ k_i= \min \{ k_1,k_2,k_3,k_4 \} $. Then
$\kb \in \HH^{(i)}_n$, which implies that $\HH_n^* \subseteq
\bigcup_{i\in \NN_4} \HH_n^{(i)}$. Since $\HH_n^{(i)} \subset \HH_n^*$
for each $i\in \NN_4$, it follows that $\HH_n^* =\bigcup_{i\in \NN_4} \HH_n^{(i)}$.

If $\kb \in \HH^{(i)}_n\cap\HH^{(j)}_n$, then $0\leq  k_i-k_j\leq 0$;
that is, $k_i=k_j$.  It follows that if $\kb \in \bigcap_{j\in J} \HH_n^{(j)}$, then
$k_i = k_j, \, \forall i,j\in J$, which implies $\bigcap_{j\in J} \HH_n^{(j)} \subseteq
\HH_n^{J}$. Since $\HH_n^J \subset \HH_n^{(j)}$ by definition, we conclude
that $\HH_n^{J} = \bigcap_{j\in J} \HH_n^{(j)}$.
\end{proof}

\begin{thm} \label{H-partial} 
For $ n \ge 0$,
\begin{align} \label{DnTheta-n}
    D_n^H(\tb) =\Theta_{n+1}(\tb) - \Theta_{n}(\tb), \qquad \hbox{where}\quad
      \Theta_{n}(\tb): = \prod_{j=1}^4 \frac{\sin \pi n t_j}{\sin \pi t_j}.
\end{align}
\end{thm}

\begin{proof}
Using the inclusion-exclusion relation of subsets, we have
$$
   D_n^H(\tb) =   \sum_{\emptyset\subset J\subseteq \NN_4} (-1)^{|J|+1}
    \sum_{\kb \in \HH_n^J}  e^{\frac{\pi i}{2}\, \kb \cdot \tb}.
$$
Fix $j\in J$, using the fact that $t_j = - \sum_{i\ne j} t_i$,  we have
\begin{align*}
\sum_{\kb \in \HH_n^J} e^{\frac{\pi i}{2}\, \kb \cdot \tb}
=&  \sum_{\kb\in \HH_n^J}   e^{\frac{\pi i}{2}\, \sum_{l\in \NN_4\setminus J}
   (k_l-k_j) t_l }
= \sum_{\kb\in \HH_n^J}   \prod_{l\in \NN_4\setminus J} e^{\frac{\pi i}{2} (k_l-k_j) t_l }.
\end{align*}
By the definition of $\HH_n^J$ and the fact that $\kb \in \HH$ implies
$k_i \equiv k_j \pmod{4}$, we obtain
\begin{align*}
\sum_{\kb \in \HH_n^J} e^{\frac{\pi i}{2}\, \kb \cdot \tb}
=&   \prod_{l\in \NN_4\setminus J} \sum_{\substack{  0\leq k_l-k_j \leq 4n \\
  k_l\equiv k_j\!\!\!\!\!\pmod{4} } }  e^{\frac{\pi i}{2} \, (k_l-k_j) t_l}\\
=&  \prod_{l\in \NN_4\setminus J} \sum_{0\leq k_l\leq n}  e^{2 \pi i\, k_l t_l}
: = \prod_{l \in \NN_4 \setminus J} K_n(t_l).
\end{align*}
Consequently, we obtain
\begin{equation}\label{eq:DnH}
D_n^H(\tb) =   \sum_{\emptyset\subset J\subseteq \NN_4} (-1)^{|J|+1}
  \prod_{l \in \NN_4 \setminus J} K_n(t_l)
   = \prod_{j\in \NN_4} K_n(t_j) - \prod_{j\in \NN_4} (K_n(t_j)-1),
\end{equation}
where the second equality is easily verified upon expanding the right hand
side explicitly. Thus, we conclude that
\begin{align*}
  D_n^H(\tb) = & \prod_{j=1}^4 \frac{\e^{2\pi i\, (n+1) t_j}-1}{e^{2\pi i\, t_j}-1 } -
   \prod_{j=1}^4 \frac{\e^{2\pi i\, (n+1) t_j}  -\e^{2\pi i t_j} }{e^{2\pi i\,\ t_j}-1 }\\
  =&\prod_{j=1}^4\frac{ e^{\pi i\, (n+1) t_j}-e^{-\pi i\, (n+1) t_j}}{e^{\pi  i\, t_j}
    -e^{-\pi i\, t_j} }  e^{\pi i\, n t_j}
   -\prod_{j=1}^4\frac{e^{\pi i\, n t_j}  - e^{-\pi i\, n t_j} }{e^{\pi i\,  t_j}-
     e^{-\pi i\, t_j} } e^{\pi i\, (n+1) t_j}\\
 = &\prod_{j=1}^4 \frac{\sin \pi (n+1) t_j}{\sin \pi t_j}
  - \prod_{j=1}^4 \frac{\sin \pi n t_j}{\sin \pi t_j},
\end{align*}
where in the last step we have used the fact that $\prod_{j=1}^4 e^{ i\alpha \, t_j}=1$,
which follows from the fact that $t_1+t_2+t_3+t_4 =0$. This completes the proof.
\end{proof}

As an immediate consequence of Theorem \ref{H-partial}, we conclude that
\begin{equation*}
           |\HH^*_n| = D_n^H(0) = (n+1)^4-n^4.
\end{equation*}

The explicit formula of the Dirichlet kernel also allows us to derive an estimate
for the norm of the partial sum  $S_n f$ in \eqref{PartialSum}.  Let $\|f\|_\infty$
denote the uniform of $f \in C(\overline{\Omega}_H)$ and let $\|S_n\|_\infty$
denote the operator norm of $S_n: C(\overline{\Omega}_H) \mapsto
C(\overline{\Omega}_H)$.

\begin{thm} \label{SnNorm} 
There is a constant $c$ independent of $f$ and $n$ such that
$$
   \|S_n \|_{\infty} \le c (\log n)^3.
$$
\end{thm}

\begin{proof}
From \eqref{PartialSum}, a standard argument shows that the norm is given by
$$
  \|S_n \|_{\infty} = \frac{1}{2}\max_{\tb\in \overline{\Omega}_H}
        \int_{\Omega_H} \left| D_n^H(\tb - \sb) \right| d\sb.
$$
To estimate the integral, we use the first equation of \eqref{eq:DnH} and the
fact that
\begin{equation} \label{eq:KnCpt}
   K_n(t) = \sum_{j=0}^n e^{2\pi i j t} = e^{\pi i n t} \frac{\sin \pi (n+1) t}{\sin \pi t},
\end{equation}
which leads to
\begin{align*}
   \|S_n\|_\infty &\,\le \frac{1}{2} \max_{\tb\in \overline{\Omega}_H}
      \sum_{\emptyset \subset J \subseteq \NN_4}
        \int_{\Omega_H} \prod_{l \in \NN_4 \setminus J} \left|K_n(t_l -s_l) \right| d\sb \\
     &\,  \le   \frac{1}{2} \max_{\tb\in \overline{\Omega}_H}
         \sum_{\emptyset \subset J \subseteq \NN_4}
        \int_{\Omega_H} \prod_{l \in \NN_4 \setminus J}
               \left|\frac{\sin \pi (n+1)(t_l-s_l)}{\sin \pi (t_l-s_l)}\right| d\sb .
\end{align*}
Since $J \ne \emptyset$, the above product contains at most three terms, and
those that contain product of three terms dominate other integrals. Consequently,
enlarging the domains of the integration and then using the periodicity
of the trigonometric function, we conclude that
\begin{align*}
  \|S_n\|_\infty \le &\ c  \int_{[-1,1]^3} \prod_{j=1}^3
                \left|\frac{\sin \pi (n+1) u_j}{\sin \pi u_j}\right| du  \\
             = &\ c  \prod_{j=1}^3 \int_{-1}^1
                    \left|\frac{\sin \pi (n+1) u_j}{\sin \pi u_j}\right| du_j
               \le c (\log n)^3,
\end{align*}
where the last step follows from the usual estimate of the integral involved.
\end{proof}

We expect that the estimate is sharp,  that is, $\|S_n\|_\infty \ge c (\log n )^3$.
To prove such a result would require a lower bound estimate of the integral of
$|D_n^H(\tb -\sb)|$ at one point in $\Omega$, likely at $\sb =0$. However, this
does not look to be an easy task as there is a sum of four terms of the same
type.


\subsection{Discrete Fourier analysis on the rhombic dodecahedron}
Using the set-up in the previous subsection, Theorem \ref{thm:2.4} in the
homogeneous coordinates becomes the following proposition

\begin{prop}\label{pro:inner}
For $n\geq 0$, define
\begin{align*}
     \langle f,\, g \rangle_n := \frac{1}{4n^3} \sum_{j\in \HH_n} f(\tfrac{\jb}{4n})
        \, \overline{g(\tfrac{\jb}{4n})},      \quad f,\, g\in C(\overline{\Omega}_H).
\end{align*}
Then
$$
\langle f,\, g \rangle = \langle f,\, g \rangle_n, \qquad f,\, g \in \CH_n.
$$
\end{prop}

The point set $\HH_n$, hence the inner product $\langle\cdot, \cdot\rangle_n$,
is not symmetric on $\Omega_H$ in the sense that it contains only part of the
points on the boundary.  Using the periodicity, however, we can show that the
inner product $\langle \cdot, \cdot \rangle_n$ is equivalent to a symmetric
discrete inner product based on $\HH^*_n$. To proceed, define
\begin{align*}
  \HH_n^{\circ} := \left\{ \jb \in \HH:  \tfrac{\jb}{4n} \in  \Omega_H^{\circ} \right\}
\end{align*}
and, recall \eqref{CK}, for $0 < i, j < i+j \leq 4$ define
\begin{align} \label{eq:Hn^ij}
  \HH_n^{i,j} := \left\{ \kb \in \HH:\  \tfrac{\kb}{4n} \in B^{i,j}  \right\},\qquad
  \HH_{n,0}^{i,j} := \left\{ \kb\in \HH:  \tfrac{\kb}{4n} \in B^{i,j}_0 \right\}.
\end{align}
Recall that $B^{i,j}$ is a boundary element of $\Omega_H$, so that $\HH_n^{i,j}$
describes those points $\jb$ in $\HH_n$ such that $\frac{\jb}{4n}$ are in
$B^{i,j}$ of $\partial \Omega_H$. Furthermore, $\HH_n^{i,j} = \{\HH_{n,0}^{i,j} \sigma:
\sigma \in \CG\}$. Using Proposition \ref{prop:BIJ}, it is easy to see that
$\HH_n^{i,j} \cap \HH_n^{k,l} = \emptyset$ if $i\ne k, j \ne l$,
\begin{align*}
\bigcup_{0<i,j<i+j\leq 4} \HH_n^{i,j}   = \HH_n^*\setminus \HH_n^{\circ}
   \qquad \hbox{and} \qquad
\bigcup_{0<i,j<i+j\leq 4} \HH_{n,0}^{i,j} = \HH_n \setminus \HH_n^{\circ}.
\end{align*}

\begin{lem} 
For $n \ge 1$, $0<i,j < i+j \le 4$,
$$
  |\HH_n^\circ| = n^4 - (n-1)^4, \qquad |\HH_n^{i,j}|= \frac{4!}{i! j! (4-i-j)!}(n-1)^{4-i-j}.
$$
\end{lem}

\begin{proof}
The first equation follows from $|\HH_n^\circ| = |\HH_{n-1}^*|=n^4 -(n-1)^4$.
The description of $B^{i,j}$ in Subsection 3.2 shows that $B^{i,j}$ has
$\frac{4!}{i! j! (4-i-j)!}$ segments and each has $(n-1)^{4-i-j}$ points, which
proves the second equation.
\end{proof}

\begin{defn} \label{defn:Sym-ipd} 
For $n \ge 0$ define the symmetric discrete inner product
\begin{align*}
 \langle f,\, g \rangle_n^* : =  \frac{1}{4n^3} \sum_{\jb\in \HH^*_n} c_{\jb}^{(n)}
 f(\tfrac{\jb}{4n})\, \overline{g(\tfrac{\jb}{4n})}, \qquad
    f,\, g \in C(\overline{\Omega}_H),
\end{align*}
where $c_{\jb}^{(n)}=1$ if $ \jb \in \HH_n^{\circ}$, and
$c_{\jb}^{(n)}=\frac{1}{\binom{i+j}{i}}$ if $ \jb \in \HH_n^{i,j}$;  more explicitly
\begin{align*}
    c_{\jb}^{(n)} =\begin{cases} 1, & \jb \in \HH_n^{\circ}, \quad
               \qquad \qquad \!\!\! (\hbox{$n^4-(n-1)^4$ points in the interior}),\\
       \frac12, & \jb \in \HH_n^{1,1}, \quad
             \quad\qquad (\hbox{$12 (n-1)^2$ points on the faces}),\\
       \frac13, & \jb \in \HH_n^{1,2}\cup \HH_n^{2,1}, \quad
             (\hbox{$2 \times 12 (n-1) $ points on the edges}),\\
       \frac14, & \jb \in \HH_n^{1,3}\cup \HH_n^{3,1}, \quad
             (\hbox{$2\times 4$  points on the vertices}),\\
       \frac16, & \jb \in \HH_n^{2,2}, \qquad  \!  \qquad (\hbox{$6$ points on the vertices}).
\end{cases}
\end{align*}
\end{defn}

It is easy to verify that $\sum_{\jb \in \HH_n^*} c_{\jb}^{(n)} = 4 n^3$, so that
$\la 1, 1 \ra_n^* =1$. We prove the following result.

\begin{thm} \label{ipdH}
 For $n\geq 0$,
\begin{align*}
   \langle f,\, g\rangle = \langle f,\,g \rangle_n = \langle f,\,g \rangle_n^*,
     \quad f,\,g\in \mathcal{H}_n.
 \end{align*}
\end{thm}

\begin{proof}
For each $\jb \in \HH_{n,0}^{i,j}$, we define
\begin{equation}\label{CSk}
  \CS_\jb =\{\kb \in \HH_n^*:  \tfrac{\kb}{4n} \equiv  \tfrac{\jb}{4n} \pmod{H} \}.
\end{equation}
It follows immediately from \eqref{eq:Bij=[B]} that
$\HH_n^{i,j} = \bigcup_{\kb\in \HH^{i,j}_{n,0}}\CS_{\kb}$.
By \eqref{eq:[B]}, $[B^{I,J}]$ is the union of $\binom{|I|+|J|}{|I|}$ components
of $B_{I,J}\sigma$. Consequently, it follows that
$|\CS_\jb| = \binom{k+l}{l}$ for $\jb \in \HH_{n,0}^{k,l}$.
Let $f$ be  an  $H$-periodic function. Then
\begin{align*}
\sum_{\jb\in \HH^*_n\setminus \HH_n^{\circ}} c_{\jb}^{(n)} f(\tfrac{\jb}{4n})
&= \sum_{0<i,k<i+k\leq 4} \frac{1}{\binom{i+k}{i}} \sum_{\jb \in \HH_n^{i,k} }
      f(\tfrac{\jb}{4n}) \\
& =  \sum_{0<i,k<i+k\leq 4} \frac{1}{\binom{i+k}{i}} \sum_{\jb \in \HH_{n,0}^{i,k} }
     \sum_{\kb \in \CS_\jb} f(\tfrac{\kb}{4n}).
\end{align*}
Since $|\CS_\jb| = \binom{i+k}{i}$ for $\jb \in \HH_{n,0}^{i,k}$,  using the
invariance of $f$, we then conclude that
$$
 \sum_{\jb\in \HH^*_n\setminus \HH_n^{\circ}} c_{\jb}^{(n)} f(\tfrac{\jb}{4n})
  =  \sum_{0<i,k<i+k\leq 4} \frac{1}{\binom{i+k}{i}} \sum_{\jb \in \HH_{n,0}^{i,k} }
       \binom{i+k}{i} f(\tfrac{\jb}{4n}) =
        \sum_{\jb\in \HH_n\setminus \HH_n^{\circ}}  f(\tfrac{\jb}{4n}).
$$
Since $c_\jb^{(n)}=1$ if $\jb \in \HH_n^\circ$, the proof is completed.
\end{proof}

The discrete inner product is closely related to cubature formula, since
Theorem \ref{ipdH} shows that the integral of $f \in \CH_n$ agrees with
the discrete sum over $\HH_n^*$.  In fact, more is true.  Let us define by
$\CT_n$ the space of generalized trigonometric polynomials,
$$
\CT_n : = \mathrm{span} \left\{\phi_{\kb}:   \kb\in \HH^*_n\right\}.
$$

\begin{thm}\label{th:cubature-H} 
For $n\ge 0$, the cubature formula
\begin{align*}
   \frac{1}{2} \int_{\Omega} f(\tb) d{\tb} = \frac{1}{4n^3} \sum_{\jb \in
     \HH_n^*}   c_{\jb}^{(n)} f(\tfrac{\jb}{4n})
\end{align*}
is exact for all $f\in \CT_{2n-1}$.
\end{thm}

\begin{proof}
If $\kb,\jb\in \HH_n$, then the definition of $\HH^*_m$ implies immediately that
$\kb-\jb\in \HH^*_{2n-1}$.  Suppose now $\jb\in \HH^*_{2n-1}$ and we may
assume that $j_1 \geq j_2 \geq j_3 \geq j_4$. 
There exists $\kb \in \HH_n^*$ such that $\kb - \jb \in \HH_n^*$.
Indeed, if $j_1 - j_3 \le  4n - 4$, we can take $\kb = (n, n, n,
-3n)$. If $j_2 - j_4 \le 4n - 4$, we can take $k = (3n, -n, -n,
-n)$. Finally, if both $j_1 - j_3 \ge 4n$ and $j_2 - j_4 \ge 4n$
then it follows from the definition of $\HH^*_{2n-1}$ that $j_1 -
j_2 \le 4n - 4$ and $j_3 -  j_4 \le 4n - 4$. In this case we can
take $\kb = (2n, 2n, -2n, -2n)$. Consequently, this shows that
\begin{align*}
  \HH^*_{2n-1} = \left\{\lb: \lb = \kb-\jb, \,  \kb,\jb\in \HH_n  \right\}.
\end{align*}
Thus if $\phi_{\jb}\in \CH_{2n-1}$, then $\jb\in \HH^*_{2n-1}$ and there exist
$\kb,\lb \in \HH_n$ such
that $\phi_\jb = \phi_{\kb} \overline{\phi_{\lb}}$. Consequently, the stated result
follows from Theorem \ref{ipdH}.
\end{proof}


\subsection{Interpolation on the rhombic dodecahedron}

For the rhombic dodecahedron, Theorem \ref{thm:interpolation} on
interpolation becomes  the following:

\begin{prop}
For $n> 0$, define
\begin{align*}
   \mathcal{I}_n f(\tb) := \sum_{\jb\in \HH_n} f(\tfrac{\jb}{4n})
           \Phi_n (\tb-\tfrac{\jb}{4n}) ,\quad
   \text{ where }\quad  \Phi_n (\tb) = \frac{1}{4n^3} \sum_{\kb \in \HH_n} \phi_k (\tb),
\end{align*}
for $f\in C(\overline{\Omega}_H)$. Then $\mathcal{I}_nf\in
\mathcal{H}_n$ and
\begin{align*}
    \mathcal{I}_nf(\tfrac{\jb}{4n}) = f(\tfrac{\jb}{4n}), \quad \forall \jb \in \HH_n.
\end{align*}
\end{prop}

Again there is a lack of symmetry in the sense that $\CI_n$ uses
only points in $\HH_n$, which contains only about half of the
boundary points. We are more interested in another interpolation
operator given below, defined over all points in $\HH_n^*$. Although
it does not interpolate at all points in $\HH_n^*$, its symmetric
form can be used to derive results on the tetrahedron in the next
section. Recall $\CS_\kb $ defined in \eqref{CSk}.

\begin{thm} \label{interp-H} 
For $n\geq 0$ and $f\in C(\overline{\Omega}_H)$, define
\begin{align*}
    \mathcal{I}_n^* f(\tb) := \sum_{\jb\in \HH^*_n} f(\tfrac{\jb}{4n}) \ell_{\jb,n} (\tb),
\end{align*}
where
\begin{align*}
  \ell_{\jb,n}(\tb) = \Phi^*_n(\tb-\tfrac{\jb}{4n}) \quad \text{ and }\quad
  \Phi^*_n(\tb) = \frac{1}{4n^3} \sum_{\kb \in \HH^*_n} c_\kb^{(n)} \phi_\kb (\tb).
\end{align*}
Then $\mathcal{I}^*_nf\in \mathcal{T}_n$ and it satisfies
\begin{align}    \label{eq:sym-interp}
         \mathcal{I}^*_n f(\tfrac{\jb}{4n}) = \begin{cases}
         f(\frac{\jb}{4n}), & \jb\in \HH^{\circ}_n,\\
         \\
         \displaystyle \sum_{\kb\in S_{\jb}}  f(\tfrac{\kb}{4n}), & \jb\in \HH^*_n
                 \setminus \HH_n^{\circ}.
\end{cases}
\end{align}
Furthermore, $\Phi_n^*(\tb)$ is a real function and it satisfies
\begin{align} \label{eq:Phi-n*}
 \Phi_n^*(\tb) = & \frac{1}{4n^3} \left[ \frac{1}{2}
         \left( D_n^H(\tb) + D_{n-1}^H(\tb) \right) -  \frac{1}{3}
              \sum_{\nu=1}^4  \frac{\sin(n-1) \pi t_\nu} {\sin \pi t_\nu}
  \sum_{\substack{j=1 \\ j \ne \nu}}^4  \cos n \pi (2 t_j + t_\nu) \right. \notag \\
   & \qquad \left. \qquad  -\frac{1}{2} \sum_{j=1}^4 \cos 2 \pi n t_j   -
       \frac13 \sum_{1 \le \mu<\nu \le 4}  \cos 2\pi  n (t_{\mu} + t_{\nu})  \right].
\end{align}
\end{thm}

\begin{proof}
By definition,
\begin{align*}
\ell_{\jb,n}(\tfrac{\kb}{4n}) =  \Phi_n^*(\tfrac{\kb-\jb}{4n}) =
    \frac{1}{4n^3} \sum_{\lb \in \HH_n^*} c_\lb^{(n)} \phi_\lb(\tfrac{\kb-\jb}{4n}).
\end{align*}
Since $\Omega_H$ tiles $\RR_H^4$, there exist $\mb,\lb\in\ZZ^4_H$ such that
$\frac{\mb}{4n} \in \Omega_H$ and $\kb-\jb=\mb+4n\lb$. Thus, by Theorem \ref{ipdH},
\begin{align*}
\ell_{\jb,n}(\tfrac{\kb}{4n}) =&\frac{1}{4n^3}\sum_{\ib\in \HH^*_n}
 c_{\ib}^{(n)} \phi_{\ib}(\tfrac{\mb}{4n}) = \frac{1}{4n^3}\sum_{\ib\in \HH^*_n}
    c_{\ib}^{(n)} \phi_{\mb}(\tfrac{\ib}{4n}) \\
 =& \langle \phi_{\mb}, \phi_{0}\rangle^*_n =
 \langle \phi_{\m}, \phi_{0}\rangle = \delta_{\mb,0}.
\end{align*}
Equivalently we can write the above equation as
\begin{align} \label{k=jmod}
\begin{split}
\ell_{\jb,n}(\tfrac{\kb}{4n}) = \langle \phi_{\kb}, \phi_{\jb}\rangle_{n}^* = \begin{cases}
 1, & \kb = \jb + 4n \lb,\, \lb\in \ZZ_\HH^4,\\
 0, & \text{otherwise},
\end{cases}
\end{split}
\end{align}
from which \eqref{eq:sym-interp} follows.

To derive the compact formula for $\ell_{\jb,n}$ we start with the obvious fact that
$\partial \HH_n^* =  \HH_n^*\setminus \HH_{n-1}^*$, so that
\begin{align*}
    \sum_{\kb \in \partial \HH_n^*} c_\kb^{(n)} \phi_\kb(\tb) & = \frac{1}{2}
       \sum_{\kb \in \HH_n^*\setminus \HH_{n-1}^*}  \phi_\kb(\tb) -
         \sum_{\kb \in \partial \HH_n^*} (\tfrac{1}{2}- c_\kb^{(n)} ) \phi_\kb(\tb) \\
   & = \frac{1}{2} \left(D_n^H(\tb) - D_{n-1}^H(\tb) \right) -
         \sum_{\kb \in \partial \HH_n^*} (\tfrac{1}{2}- c_\kb^{(n)} ) \phi_\kb(\tb).
\end{align*}
Since $\HH_n^*= \HH_n^\circ \cup \partial \HH_n^*$, we then derive from the
decomposition of $\partial \HH_n^*$ into $\HH_n^{i,j}$ and the values of
$c_\kb^{(n)}$ that
\begin{align} \label{eq:Phi*}
  \Phi^*_n(\tb) = \frac{1}{4n^3}\bigg[&\frac12(D_n^H(\tb)+D_{n-1}^H(\tb))
   - \frac16 \sum_{k\in \HH_n^{1,2}\cup \HH_n^{2,1}} \phi_\kb(\tb)\\
   &- \frac14 \sum_{k\in \HH_n^{1,3}\cup \HH_n^{3,1}} \phi_\kb(\tb)
   - \frac1{3} \sum_{k\in \HH_n^{2,2} } \phi_\kb(\tb)
    \bigg].\notag
\end{align}
Let us define $\HH_n^{I,J}: = \{\kb \in \HH: \tfrac{\kb}{4n} \in B_{I,J}\}$ for $I, J
\subset \NN_4$ and also define $\left[\HH_n^{I,J}\right]: = \{\kb \in \HH:
\tfrac{\kb}{4n} \in [B_{I,J}]\}$. It follows from \eqref{eq:Bij=[B]}, \eqref{eq:Hn^ij}
and Lemma \ref{congruence} that
$$
\HH_n^{i,j} = \bigcup_{I,J\in \CK_0^{i,j}} \left[\HH_n^{I,J} \right] \qquad\hbox{and}\qquad
  \left[\HH_n^{I,J} \right] =  \bigcup_{\sigma\in \mathcal{G}_{I\cup J}} \HH_n^{I,J}\sigma.
$$
In particular, by \eqref{eq:B12}, we have
\begin{align*}
 &\HH_n^{1,2}   = [\HH_n^{\{1\},\{2,3\}}] \cup  [\HH_n^{\{1\},\{2,4\}}]
   \cup [\HH_n^{\{1\},\{3,4\}}]\cup  [\HH_n^{\{2\},\{3,4\}}],\\
& \HH_n^{2,1} =  [\HH_n^{\{1,2\},\{3\}}]\cup [\HH_n^{\{1,2\},\{4\}}]\cup
   [\HH_n^{\{1,3\},\{4\}}]  \cup [\HH_n^{\{2,3\},\{4\}}].
\end{align*}
By \eqref{[B{1,23}]}, $ [\HH_n^{\{1\},\{2,3\}} ] = \HH_n^{\{1\},\{2,3\}}
 \cup  \HH_n^{\{2\},\{1,3\}} \cup \HH_n^{\{3\},\{1,2\}}$. Furthermore, it follows from
 \eqref{eq:Bij} that $[\HH_n^{\{1\},\{2,4\}}]= [\HH_n^{\{1\},\{2,3\}}]\sigma_{34}$,
 $[\HH_n^{\{1\},\{3,4\}}]=  [\HH_n^{\{1\},\{2,3\}}]\sigma_{24}$,
$[\HH_n^{\{2\},\{3,4\}}] = [\HH_n^{\{1\},\{2,3\}}]\sigma_{12}\sigma_{24}$.
Using the explicit formulas in \eqref{B_{1,{2,3}}}, it is easy to see that
\begin{align*}
   \HH_n^{\{1\},\{2,3\}} = \left\{ (j+2n,j-2n,j-2n,2n-3j): 1 \le j \le n-1 \right\}.
\end{align*}
Consequently, using $t_1+t_2+t_3+t_4 = 0$, it follows readily that
\begin{align*}
   \sum_{\kb \in [\HH_n^{\{1\},\{2,3\}}]} \phi_\kb(\tb) & =
       \sum_{\kb \in \HH_n^{\{1\},\{2,3\}}} \phi_\kb(\tb) +
          \sum_{\kb \in \HH_n^{\{2\},\{1,3\}}} \phi_\kb(\tb) +
             \sum_{\kb \in \HH_n^{\{3\},\{1,2\}}} \phi_\kb(\tb) \\
    &  =    \sum_{j=1}^{n-1} e^{-2 \pi i j t_4}  \left(
      e^{2 n \pi i (t_1+t_4)} + e^{2 n \pi i (t_2+t_4)} +e^{2 n \pi i (t_3+t_4)} \right) \\
    & = \frac{\sin (n-1)\pi t_4}{\sin \pi t_4} \left(e^{n \pi i (2 t_1+t_4)} +
          e^{n \pi i (2 t_2+t_4)}+ e^{n \pi i (2 t_3+t_4)} \right),
\end{align*}
where in the last step the sum is evaluated using \eqref{eq:KnCpt}. The
explicit formulas for other components of $\sum_{\kb\in \HH_n^{1,2}} \phi_\kb$
follow from the above expression by permuting the variables. In a similar manner, we
have $[\HH_n^{\{1,2\},\{3\}}] = \HH_n^{\{1,2\},\{3\}}
 \cup  \HH_n^{\{1,3\},\{2\}} \cup \HH_n^{\{2,3\},\{1\}}$, and, using \eqref{eq:Bij} again,
 $[\HH_n^{\{1,2\},\{4\}}]= [\HH_n^{\{1,2\},\{3\}}]\sigma_{34}$,  $[\HH_n^{\{1,3\},\{4\}}]
 =[\HH_n^{\{1,2\},\{3\}}]\sigma_{23}\sigma_{34}$,  $[\HH_n^{\{2,3\},\{4\}}] =
 [\HH_n^{\{1,2\},\{3\}}]\sigma_{13}\sigma_{34}$. Moreover, we also have
$$
\HH_n^{\{1,2\},\{3\}} = \left\{ (2n-j, 2n-j, -2n-j, - 2n+3j): 1 \le j \le n-1 \right\}.
$$
Thus, using $t_1+t_2+t_3+t_4 = 0$, we can deduce as before that
\begin{align*}
   \sum_{\kb \in [\HH_n^{\{1,2\},\{3\}}]} \phi_\kb(\tb) & =
       \sum_{\kb \in \HH_n^{\{1,2\},\{3\}}} \phi_\kb(\tb) +
          \sum_{\kb \in \HH_n^{\{1,3\},\{2\}}} \phi_\kb(\tb) +
             \sum_{\kb \in \HH_n^{\{2,3\},\{1\}}} \phi_\kb(\tb) \\
     & = \frac{\sin (n-1)\pi t_4}{\sin \pi t_4} \left(e^{- n \pi i (2 t_1+t_4)} +
          e^{- n \pi i (2 t_2+t_4)}+ e^{- n \pi i (2 t_3+t_4)} \right),
\end{align*}
from which the explicit formulas of other components of
$\sum_{\kb\in \HH_n^{2,1}}\phi_\kb$ follow from permuting the variables.

Putting the sums over $\HH_n^{1,2}$ and $\HH_n^{2,1}$ together, we obtain
\begin{align*}
   \sum_{\kb\in \HH_n^{1,2}\cup \HH_n^{2,1}} \phi_k(\tb)  =
       2 \sum_{\nu=1}^4  \frac{\sin(n-1) \pi t_\nu}{\sin \pi t_\nu}
          \sum_{\substack{j=1 \\ j \ne \nu}}^4  \cos n \pi (2 t_j + t_\nu).
\end{align*}

Using \eqref{eq:B13}, it is easy to see that
$\HH_n^{1,3} = \{(n,n,n,-3n)\sigma: \sigma \in \CG\}$,
$\HH_n^{2,2} = \{(2n,2n,-2n,-2n)\sigma: \sigma \in \CG\}$ and
$\HH_n^{3,1} = \{(3n,-n,-n,-n)\sigma: \sigma \in \CG\}$, from which it
follows that
\begin{align*}
   \sum_{\kb\in \HH^{1,3}_n\cup \HH^{3,1}_n} \phi_k(\tb) =    \sum_{j=1}^4
   \big( e^{2\pi i n t_j} + e^{-2\pi i n t_j} \big)
   =  2 \sum_{j=1}^4 \cos 2 \pi  n t_j.
\end{align*}
Furthermore, using $t_1+t_2+t_3+t_4 =0$, it is easy to see that
\begin{align*}
 \sum_{\kb\in \HH^{2,2}_n} \phi_\kb(\tb)
  =  \sum_{1 \le \mu<\nu \le 4}  e^{2\pi i n (t_{\mu} + t_{\nu})}
  =   \sum_{1 \le \mu<\nu \le 4}  \cos 2\pi  n (t_{\mu} + t_{\nu}).
\end{align*}

Putting these terms into \eqref{eq:Phi*} completes the proof.
\end{proof}

The compact formula of the interpolation function allows us to estimate the
operator norm of $I_n^*$, which is usually referred to as the Lebesgue constant.

\begin{thm} \label{LebesgueH}
Let $\|I_n^*\|_\infty$ denote the operator norm of $I_n^*: C(\overline{\Omega}_H)
 \mapsto C(\overline{\Omega}_H)$. Then there is a constant $c$, independent of
$n$, such that
$$
       \|I_n^*\|_\infty  \le c (\log n)^3.
$$
\end{thm}

\begin{proof}
A standard procedure shows that
$$
   \|I_n^*\|_\infty = \max_{\tb \in \overline{\Omega}_H}  \sum_{\kb \in \HH_n^*}
         \left|\Phi_n^*(\tb-\tfrac{\kb}{4n})\right|.
$$
Using the compact formula of $\Phi_n^*$ in Theorem \ref{interp-H}, it is
easy to see that it suffices to prove that
$$
\frac{1}{4n^3}  \max_{\tb\in \overline{\Omega}_H}  \sum_{\kb \in \HH_n^*}
         \left|D_n^H(\tb-\tfrac{\kb}{4n})\right| \le c (\log n)^3, \qquad n \ge 0.
$$
Furthermore, as in the proof of Theorem \ref{SnNorm}, the formula of $D_n^H$
in \eqref{eq:DnH} shows that our main task is to establish the estimate
$$
I_{\{1,2,3\}}:= \frac{1}{4n^3}\max_{\tb\in \overline{\Omega}_H}  \sum_{\kb \in \HH_n^*}
 \left| K_n(t_1 -\tfrac{k_1}{4n})K_n(t_2 -\tfrac{k_2}{4n})K_n(t_3 -\tfrac{k_3}{4n})\right|
    \le c (\log n)^3,
$$
and three other similar estimates $I_{\{1,2,4\}}$, $I_{\{1,3,4\}}$ and $I_{\{2,3,4\}}$, respectively. Enlarging the domain $\HH_n^*$
to $\{\kb \in \ZZ_H^4: -4n \le k_i \le 4n, \, k_i = 0 \pmod{4}, \,1 \le i \le 3\}$, we see
that
\begin{align*}
   I_{\{1,2,3\}} & \le \frac{1}{4n^3}  \max_{t \in {[-1,1]^3}}  \sum_{k_1=-n}^{n}
         \sum_{k_2=-n}^{n} \sum_{k_3= -n}^{n}
      \left| K_n(t_1 -\tfrac{k_1}{n})K_n(t_2 -\tfrac{k_2}{n})K_n(t_3 -\tfrac{k_3}{n})\right|\\
   &   \le    \frac14  \max_{t \in {[-1,1]}} \left (\frac{1}{n}
        \sum_{k = -n}^{n} \left | \frac{\sin (n+1) \pi (t -\tfrac{k}{n})}
          {\sin\pi (t -\tfrac{k}{n})} \right|  \right)^3   \le c (\log n)^3,
\end{align*}
where the last step follows from the standard estimate of one variable
(cf. \cite[Vol. II, p. 19]{Z}).
\end{proof}

Again we expect that the estimate is sharp, that is, $\|I_n^*\| \ge c (\log n)^3$; and
the problem is again that there is a sum of four terms of the same type.


\section{Discrete Fourier analysis on the Tetrahedron}
\setcounter{equation}{0}

Considering functions invariant under the isometrics of the fcc lattice,
the discrete Fourier analysis on the dodecahedron in the previous section
can be carried over to the analysis on the tetrahedron.

\subsection{Generalized sine and cosine functions}
The fcc lattice is the root lattice of the reflection group $\A_3$ \cite[Chapt. 4]{CS}.
Under the homogeneous coordinates, the group $\A_3$ is generated by the
reflections $\{\sigma_{ij}: 1 \le i < j \le 4\}$, where $\sigma_{ij}$ is the reflection
defined by
$
    \tb \sigma_{ij} = \tb - 2 \frac{\la \tb, \eb_{i,j}\ra}{\la \eb_{i,j},\eb_{i,j}\ra} \eb_{i,j}
$
with $\eb_{i,j} = e_i- e_j$ as before. Thus $\A_3$ is the permutation group in the
previous section. Denote the identity element in $\A_3$ by $1$. It is easy to see
that we have
$$
  \sigma_{ij}^2 =1, \qquad \sigma_{ij}\sigma_{jk}\sigma_{ij} = \sigma_{ik},
    \quad i,j,k \in \NN_4.
$$
For $\sigma \in \G = \A_3$, let $|\sigma|$ denote the number
of inversions in $\sigma$. The group $\G$ is naturally divided into two parts,
$\G^+: = \left\{\sigma\in \mathcal{G}:\  |\sigma| \equiv 0 \pmod{2} \right\}$ of
elements with even inversions, and
$\G^{-}: = \left\{\sigma\in \mathcal{G}:\  |\sigma| \equiv 1 \pmod{2} \right\}$ of
elements with odd inversions. Writing it out explicitly, we have
\begin{align*}
 \G^+ = \big\{ & 1,  \sigma_{12}\sigma_{13}, \sigma_{13}\sigma_{12},
    \sigma_{12}\sigma_{14}, \sigma_{14}\sigma_{12}, \sigma_{13}\sigma_{14}, \\
 & \sigma_{14}\sigma_{13}, \sigma_{23}\sigma_{24},
    \sigma_{24}\sigma_{23}, \sigma_{12}\sigma_{34}, \sigma_{13}\sigma_{24},
    \sigma_{14}\sigma_{23}\big\}\\
 \G^- = \big\{ & \sigma_{12}, \sigma_{13}, \sigma_{14}, \sigma_{23}, \sigma_{24},
  \sigma_{34}, \sigma_{12}\sigma_{13} \sigma_{14},
  \sigma_{12}\sigma_{14}\sigma_{13}, \\
 & \qquad  \sigma_{13}\sigma_{12}\sigma_{14}, \sigma_{13}\sigma_{14}\sigma_{12},
   \sigma_{14}\sigma_{12}\sigma_{13}, \sigma_{14}\sigma_{13}\sigma_{12}  \big\}.
\end{align*}
The action of $\sigma \in \G$ on the function $f: \RR_H^4 \mapsto \RR$ is defined
by $\sigma f (\tb):= f(\tb \sigma)$. A function $f$ in homogeneous coordinates is
called {\it invariant} under $\G$ if $\sigma f = f$  for all $\sigma \in \G$, and it is
called {\it anti-invariant} under $\G$ if $\sigma f = \rho(\sigma)f$ with
$\rho(\sigma) =1$ if $\sigma \in \G^+$ and $\rho(\sigma) = -1$ if $\sigma \in \G^-$.

The following proposition follows immediately from the definition.

\begin{prop}
Define two operator $\mathcal{P}^{+}$ and $\mathcal{P}^{-}$ acting on $f(\tb)$ by
\begin{align} \label{eq:Ppm}
\mathcal{P}^{\pm}f(\tb) := \frac1{24} \bigg[ \sum_{\sigma\in \mathcal{G}^{+}}
   f(\tb \sigma)\pm \sum_{\sigma\in \mathcal{G}^{-}} f(\tb \sigma) \bigg].
\end{align}
Then the operators $\mathcal{P}^{+}$ and $\mathcal{P}^{-}$ are projections from
the class of $H$-periodic functions onto the class of invariant, and respectively
anti-invariant functions.
\end{prop}

Applying the operators $\CP^\pm$ to $\phi_\kb(\tb) = e^{\frac{\pi
i}{2} \kb \cdot \tb}$ gives basic invariant and anti-invariant
functions, which we denote by $\TC_\kb$ and $\TS_\kb$, respectively,
as they are analogues of cosine and sine functions. We formerly
define them as follows.

\begin{defn} 
For $\kb \in \HH$ define
\begin{align*}
&\TC_{\kb}(\tb):=\mathcal{P}^{+}\phi_{\kb}(\tb) = \frac1{24\
} \bigg[ \sum_{\sigma\in \mathcal{G}^{+}}\phi_{\kb}(\tb \sigma)
+ \sum_{\sigma\in \mathcal{G}^{-}} \phi_{\kb}(\tb \sigma) \bigg],\\
&\TS_{\kb}(\tb):= - \mathcal{P}^{-}\phi_{\kb}(\tb) = - \frac1{24}
 \bigg[ \sum_{\sigma\in \mathcal{G}^{+}}\phi_{\kb}(\tb \sigma)
- \sum_{\sigma\in \mathcal{G}^{-}} \phi_{\kb}(\tb \sigma) \bigg],
\end{align*}
and call them \emph{generalized cosine} and \emph{generalized sine},
respectively.
\end{defn}

Evidently $\TC_{\kb}$ is invariant and $\TS_{\kb}$ is anti-invariant.
The rhombic dodecahedron is invariant under our group $\G$ of order 24; its
fundamental domain is a tetrahedron. For invariant functions, we can make use
of symmetry to translate results on the rhombic dodecahedron to one of its 24
tetrahedrons. We shall choose our reference tetrahedron as
\begin{equation}\label{triangle}
\triangle : = \{x \in \RR^3: 0 \le x_3 \pm x_2, x_2 \pm x_1  \le 1\}.
\end{equation}
In the homogeneous coordinates, by \eqref{coordinate}, this tetrahedron becomes
\begin{align*}
   \triangle_H := \left\{ \tb \in \RR_H^4:
               0\leq t_1-t_2,t_2-t_3,t_3-t_4,t_1-t_4 \leq 1\right\}.
\end{align*}
See Figure 4.1 below, in which coordinates of the corners are given
in both $\RR^3$ coordinates and homogeneous coordinates in $\RR^4$.

\begin{figure}[htb]
\centering
\includegraphics[width=0.6\textwidth]{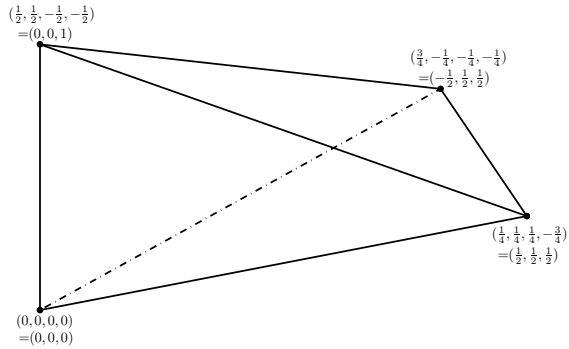}
\caption{Reference tetrahedron.}
\end{figure}

When $\TC_\kb$ are restricted to the tetrahedron $\triangle_H$, we only need
to consider a subset of $\kb\in \HH$. In fact,  it is easy to see that
\begin{align}
\label{eq:TCs}
     \TC_{\kb \sigma}(\tb) = \TC_{\kb}(\tb\sigma) = \TC_{\kb}(\tb)
\end{align}
for $\tb\in\triangle_H$ and $\sigma\in\G$. Thus, we can restrict $\kb$
to the index set
\begin{align*}
\Lambda:= \left\{ \kb\in \HH: k_1 \geq k_2 \geq k_3\geq k_4\right\}
\end{align*}
when we consider $\TC_\kb$. As for $\TS_{\kb}$, it is easy to see that
$\TS_{\kb\sigma}(\tb)= \TS_{\kb}(\tb\sigma) = \TS_{\kb}(\tb)$ for
$\sigma\in\mathcal{G}^{+}$  and  $\TS_{\kb\sigma}(\tb)= \TS_{\kb}(\tb\sigma)
= -\TS_{\kb}(\tb) $ for $\sigma\in\mathcal{G}^{-}$. In particular,
$\TS_{\kb}(\tb)=0$ whenever two or more components of $\kb$ are
equal. Thus $\TS_{\kb}$ are defined only for $\kb\in \Lambda^{\circ}$,
where
\begin{align*}
\Lambda^{\circ}:= \left\{ \kb\in \HH: k_1>k_2>k_3>k_4 \right\},
\end{align*}
which is the set of the interior points of $\Lambda$. Since $\kb \in \HH$ implies
that $k_i  \in \ZZ$ and $k_1 + k_2+k_3+k_4 =0$, the points in $\Lambda$ lies in
a three dimensional wedge. To describe the points on the boundary of $\Lambda$,
we further define
\begin{align*}
 \Lambda^{f} & := \left\{ (\kb \in \HH: k_1=k_2>k_3>k_4 \text{or}
                 k_1>k_2=k_3>k_4 \text{or} k_1>k_2>k_3=k_4  \right\},\\
     \Lambda^{e,2} & := \left\{(k,k,k,-3k), (3k,-k,-k,-k):\   k>0 \right\}, \\
    \Lambda^{e,1} & := \left\{(2k,2k,-2k,-2k):\   k>0 \right\},  \quad
     \Lambda^v  := \{ (0,0,0,0)\}.
\end{align*}
Then evidently
$ \Lambda \setminus \Lambda^\circ =  \Lambda^f \cup
       \Lambda^{e,1} \cup \Lambda^{e,2} \cup \Lambda^v.$

Let $\kb\mathcal{G}$ denote the orbit of $\kb $ under $\mathcal{G}$, that is,
$\kb \G:= \left\{\kb\sigma :\ \sigma\in \mathcal{G}\right\}$. Then,
for $\kb, \jb \in \Lambda $, $\kb\mathcal{G}\cap \jb\mathcal{G} = \emptyset$
whenever $\kb\neq  \jb$. Furthermore, it is easy to see that
\begin{align}
\label{eq:TC-redef}
 \TC_{\kb} (\tb) = \frac{1}{|\kb\mathcal{G}|}
    \sum_{\jb\in  \kb \mathcal{G}} \phi_{\jb},\qquad   |\kb\mathcal{G}| =
  \begin{cases}
        24, & \kb\in \Lambda^{\circ},\\
        12, & \kb\in \Lambda^{f},\\
        6, & \kb\in \Lambda^{e,1},\\
        4, & \kb\in \Lambda^{e,2},\\
        1, & \kb=0\in \Lambda^{v}.
  \end{cases}
\end{align}

We define an inner product on $\triangle_H$ by
\begin{align*}
\langle f, g \rangle_{\triangle_H}:=
\frac{1}{|\triangle_H|} \int_{\triangle_H} f(\tb) \overline{g({\tb})} d\tb
   = 12\int_{\triangle_H} f(\tb) \overline{g({\tb})} dt_1dt_2dt_3.
\end{align*}
If $f\bar g$ is invariant under $\mathcal{G}$, then it follows immediately that
$\langle f, g \rangle = \langle f, g \rangle_{\triangle_H}$. Furthermore, the
generalized cosine and sine functions are orthogonal with respect to this
inner product.

\begin{prop} 
For $\kb,\jb\in \Lambda$,
 \begin{align}
 \label{eq:orth-TC}
   \langle \TC_{\kb}, \TC_{\jb}\rangle_{\triangle_H}  =
     \frac{\delta_{\kb,\lb}}{|\kb\mathcal{G}|}
   = \delta_{\kb,\lb}    \begin{cases} 1, & \kb = 0,\\
   \frac14, &  \kb\in \Lambda^{e, 2},\\
   \frac16, &  \kb\in \Lambda^{e, 1},\\
   \frac1{12},   & \kb\in \Lambda^{f}\\
   \frac{1}{24}, & \kb\in \Lambda^{\circ};
   \end{cases}
 \end{align}
for $\kb, \jb \in \Lambda^{\circ}$,
 \begin{align}
 \label{eq:orth-TS}
 \langle \TS_{\kb}, \TS_{\jb}\rangle_{\triangle_H}  = \frac1{24}\delta_{\kb,\jb}.
 \end{align}
\end{prop}

\begin{proof}
Both of these relations follow from the identity $\langle f, g \rangle =
\langle f, g \rangle_{\triangle_H}$ for invariant functions. For \eqref{eq:orth-TC},
the invariance is evident and we only have to use the orthogonality of $\phi_\kb$
in Proposition \ref{prop:ortho-H} and \eqref{eq:TC-redef}. For \eqref{eq:orth-TS},
we use the fact that $\TS_{\kb}(\tb)\overline{\TS_{\jb}(\tb)}$ is invariant under $\G$
and the orthogonality of $\phi_\kb$ on $\Omega_H$.
\end{proof}

We can derive more explicit formulas for the generalized cosine and sine functions
by making use of the homogeneous coordinates. For example, we have
\begin{align*}
& \phi_{\kb}({t_1,t_2,t_3,t_4})-\phi_{\kb}({t_2,t_1,t_3,t_4})-\phi_{\kb}({t_1,t_2,t_4,t_3})+\phi_{\kb}({t_2,t_1,t_4,t_3})\\
& =\big[ e^{\frac{\pi i}{2}(k_1t_1+k_2t_2)}-e^{\frac{\pi i}{2}(k_1t_2+k_2t_1)}
 \big]
\big[ e^{\frac{\pi i}{2}(k_3t_3+k_4t_4)}-e^{\frac{\pi i}{2}(k_3t_4+k_4t_3)}
\big]\\
 & = 2i e^{\frac{\pi i}{4}(k_1+k_2)(t_1+t_2)} \sin\tfrac{\pi (k_1-k_2)(t_1-t_2)}{4}
 \cdot 2i  e^{\frac{\pi i}{4}(k_3+k_4)(t_3+t_4)}  \sin\tfrac{\pi (k_3-k_4)(t_3-t_4)}{4}\\
& =-4e^{\frac{\pi i}{4}[(k_1+k_2)(t_1+t_2)+(k_3+k_4)(t_3+t_4)]}
   \sin \tfrac{\pi (k_1-k_2)(t_1-t_2)}{4}
   \sin \tfrac{\pi(k_3-k_4)(t_3-t_4)}{4}\\
&= \phi_{\kb}({t_3,t_4,t_1,t_2})-\phi_{\kb}({t_4,t_3,t_1,t_2})-
      \phi_{\kb}({t_3,t_4,t_2,t_1})+\phi_{\kb}({t_4,t_3,t_2,t_1}).
\end{align*}
Similarly or by permuting the variables, we also have,
\begin{align*}
&\phi_{\kb}({t_1,t_3,t_4,t_2})-\phi_{\kb}({t_3,t_1,t_4,t_2})-\phi_{\kb}({t_1,t_3,t_2,t_4})+\phi_{\kb}({t_3,t_1,t_2,t_4})\\
& =\phi_{\kb}({t_4,t_2,t_1,t_3})-\phi_{\kb}({t_4,t_2,t_3,t_1})-\phi_{\kb}({t_2,t_4,t_1,t_3})+\phi_{\kb}({t_2,t_4,t_3,t_1})\\
& = -4e^{\frac{\pi i}{4}[(k_1+k_2)(t_1+t_3)+(k_3+k_4)(t_2+t_4)]}
\sin \tfrac{\pi (k_1-k_2)(t_1-t_3)}{4} \sin \tfrac{\pi (k_3-k_4)(t_4-t_2)}{4},
\end{align*}
and
\begin{align*}
 & \phi_{\kb}({t_1,t_4,t_2,t_3})-\phi_{\kb}({t_4,t_1,t_2,t_3})-\phi_{\kb}({t_1,t_4,t_3,t_2})+\phi_{\kb}({t_4,t_1,t_3,t_2})\\
 & =\phi_{\kb}({t_2,t_3,t_1,t_4})-\phi_{\kb}({t_2,t_3,t_4,t_1})-\phi_{\kb}({t_3,t_2,t_1,t_4})+\phi_{\kb}({t_3,t_2,t_4,t_1})\\
& = -4 e^{\frac{\pi i}{4}[(k_1+k_2)(t_1+t_4)+(k_3+k_4)(t_2+t_3)]}
\sin \tfrac{\pi (k_1-k_2)(t_1-t_4)}{4} \sin \tfrac{\pi (k_3-k_4)(t_2-t_3)}{4}.
\end{align*}
Consequently, using the homogeneous relations $t_1+t_2+t_3+t_4=0$ and
$k_1+k_2+k_3+k_4=0$, we can then deduce that
\begin{align}
\label{eq:TS}
   \TS_{\kb}(\tb)
   =&\tfrac{1}{3} e^{\frac{\pi i}{2}(k_1+k_2)(t_1+t_2)} \sin{\tfrac{\pi }{4}(k_1-k_2)(t_1-t_2)} \sin{\tfrac{\pi }{4}(k_3-k_4)(t_3-t_4)}\\
   +&\tfrac{1}{3} e^{\frac{\pi i}{2}(k_1+k_2)(t_1+t_3)} \sin{\tfrac{\pi }{4}(k_1-k_2)(t_1-t_3)} \sin{\tfrac{\pi }{4}(k_3-k_4)(t_4-t_2)}\notag\\
   +&\tfrac{1}{3} e^{\frac{\pi i}{2}(k_1+k_2)(t_1+t_4)} \sin{\tfrac{\pi }{4}(k_1-k_2)(t_1-t_4)} \sin{\frac{\pi }{4}(k_3-k_4)(t_2-t_3)}\notag.
\end{align}
In a similar way, we obtain
\begin{align}
\label{eq:TC}
   \TC_{\kb}(\tb)
   =&\tfrac{1}{3} e^{\frac{\pi i}{2}(k_1+k_2)(t_1+t_2)} \cos{\tfrac{\pi }{4}(k_1-k_2)(t_1-t_2)} \cos{\tfrac{\pi }{4}(k_3-k_4)(t_3-t_4)}\\
   +&\tfrac{1}{3} e^{\frac{\pi i}{2}(k_1+k_2)(t_1+t_3)} \cos{\tfrac{\pi }{4}(k_1-k_2)(t_1-t_3)} \cos{\tfrac{\pi }{4}(k_3-k_4)(t_4-t_2)}\notag\\
   +&\tfrac{1}{3} e^{\frac{\pi i}{2}(k_1+k_2)(t_1+t_4)} \cos{\tfrac{\pi }{4}(k_1-k_2)(t_1-t_4)} \cos{\tfrac{\pi }{4}(k_3-k_4)(t_2-t_3)}.\notag
\end{align}
Permuting variables $t_1,t_2,t_3,t_4$ lead to other representations of   $\TS_\kb$
and $\TC_\kb$.


\subsection{Discrete inner product on the tetrahedron}

Using the fact that $\TC_k$ and $\TS_k$ are invariant and anti-invariant
under $\G$ and the orthogonality of $\phi_\kb$ with respect to the symmetric
inner product $\langle \cdot, \cdot \rangle_n^*$, we can deduce a discrete
orthogonality for the generalized cosine and sine functions. For this purpose,
we define
\begin{align} \label{eq:Lambda_n}
\Lambda_n : = \HH_n^* \cap \Lambda =
     \left\{\kb\in \HH:  k_4\leq k_3\leq k_2\leq k_1\leq k_4+4n  \right \}.
\end{align}
The point set with $n=4$ and the region is depicted in Figure 4.1.

\begin{figure}[htb]
\centering
\includegraphics[width=0.6\textwidth]{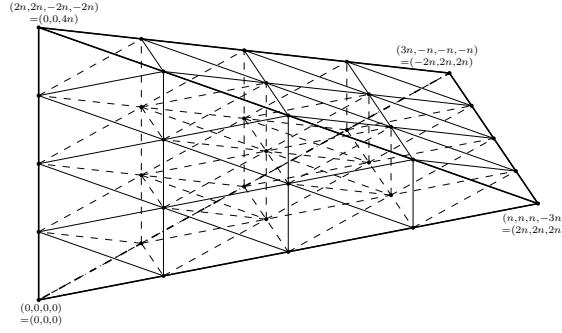}
\caption{$\Lambda_n$ with $n =4$}
\end{figure}

By the definition of $\HH_n^*$,  the set
$\{\tfrac{\kb}{4n}: \kb \in \Lambda_n\}$ contains points inside $\triangle_H$.
We will also need notation for points on the boundary of $\triangle_H$, which are
defined as
\begin{align} \label{eq:Lambda-n2}
\begin{split}
&\Lambda^{\circ}_n: = \HH_n^\circ \cap \Lambda^\circ, \qquad
\Lambda_n^{f} : = \left(\Lambda^f \cap \HH_n^\circ \right) \cup
     \left(\Lambda^\circ \cap \HH_n^{1,1}\right),  \\
&\Lambda^{e,1}_n: = \left(\Lambda^{e,1} \cap \HH_n^\circ \right) \cup
     \left(\Lambda^f \cap \HH_n^{1,1}\right),\\
& \Lambda^{e,2}_n: = \left(\Lambda^{e,2} \cap \HH_n^\circ \right) \cup
 \left(\Lambda^f \cap \HH_n^{1,2}\right)\cup \left (\Lambda^f \cap \HH_n^{2,1}\right), \\
& \Lambda^{v}_n: = \{0\} \cup \left( \Lambda^{e,1} \cap \HH_n^{2,2} \right)
         \cup \left( \Lambda^{e,2} \cap \HH_n^{1,3} \right)\cup \left
         (\Lambda^{e,2}
         \cap \HH_n^{3,1} \right),
\end{split}
\end{align}
corresponding to the set of the interior points, the set of the points on faces,
the set of points on two type of edges, and the set of vertices, respectively. More precisely, these sets are given explicitly by
\begin{align*}
&\Lambda_n^{\circ} = \left\{ \kb\in \HH:  k_4<k_3<k_2<k_1<k_4+4n \right\},\\
&\Lambda_n^{f}  = \big\{ \kb\in \HH:\   k_4<k_3<k_2<k_1=k_4+4n \, \text{ or } \,  k_4<k_3<k_2=k_1<k_4+4n \\
       &\qquad \qquad  \text{ or }  k_4<k_3=k_2<k_1<k_4+4n   \, \text{ or }  \, k_4=k_3<k_2<k_1<k_4+4n  \big\},\\
&\Lambda_n^{e,1}   = \left\{ (2k,2k,-2k,-2k), (2k+n, n-2k, n-2k, 2k-3n):\  0<k<n\right\},\\
&\Lambda_n^{e,2}   = \big\{ (k,k,k,-3k), (3k,-k,-k,-k), (n+k, n+k, n-3k, k-3n),\\
       &\qquad\qquad\qquad \qquad \qquad   (3n-k, 3k-n, -n-k, -n-k):\  0<k<n\big\},\\
&\Lambda_n^{v}
= \left\{(0,0,0,0),\ (2n,2n-2n,-2n),\ (3n,-n,-n,-n),\ (n,n,n,-3n)\right\}.
\end{align*}

We denote by $\CTC_n$ and $\CTS_n$ the spaces of the
trigonometric polynomials
\begin{align*}
\CTC_n:=\mathrm{span}\left\{\TC_{\kb}:\ \kb\in \Lambda_n  \right\}, \qquad
\CTS_n:=\mathrm{span}\left\{\TS_{\kb}:\ \kb\in \Lambda_n^{\circ} \right\},
\end{align*}
respectively. We define a discrete inner product $\langle \cdot,\cdot
\rangle_{\triangle,n}$ by
\begin{align*}
    \langle f,g \rangle_{\triangle,n} = \frac{1}{4n^3} \sum_{\jb\in \Lambda_n}
    \lambda_{\jb}^{(n)} f(\frac{\jb}{4n}) \overline{g(\frac{\jb}{4n})},
\end{align*}
where
\begin{align}  \lambda_{ \jb}^{(n)} := \begin{cases}
          24, & \jb\in\Lambda^{\circ}_n,\\
          12, & \jb\in\Lambda^{f}_n,\\
           6,  &  \jb\in\Lambda^{e,1}_n,\\
           4,  &  \jb\in\Lambda^{e,2}_n,\\
           1,  &  \jb\in\Lambda^{v}_n. \end{cases}
\end{align}

\begin{thm} \label{thm:4.4}
For $f \bar{g} \in \CTC_{2n-1}$,
\begin{align}  \label{eq:orth-Tn}
   \la f, g \ra_{\triangle_H} =  \la f, g \ra_{\triangle,n}.
\end{align}
Moreover, the following cubature formula is exact for all $f\in \CTC_{2n-1}$,
\begin{align} \label{eq:cubature-T}
   \frac{1}{|\triangle_H|} \int_{\triangle_H}f(\tb) d\tb =
       \frac{1}{4n^3} \sum_{\jb\in \Lambda_n} \lambda_{\jb}^{(n)} f(\tfrac{\jb}{4n}).
\end{align}
In particular,
\begin{align}  \label{eq:orth-TCn}
   \langle \TC_{\kb},\TC_{\jb} \rangle_{\triangle,n} =
    \frac{ \delta_{\kb,\jb}}{\lambda_{\kb}^{(n)}}, \quad \kb,\jb \in \Lambda_n.
\end{align}
\end{thm}

\begin{proof}
We shall deduce the result from Theorem \ref{th:cubature-H}. Let $f$ be
a function invariant under $\G$. Recall the coefficients $c_\jb^{(n)}$ in the
symmetric inner product defined in Definition \ref{defn:Sym-ipd}.
Taking into consideration of the orbits of the points in various regions, we obtain
\begin{align*}
\sum_{\jb\in \HH_n^{\circ}} c_{\jb}^{(n)} f(\tfrac{\jb}{4n}) =
   \sum_{\jb\in \HH_n^{\circ}} f(\tfrac{\jb}{4n})
& = 24  \sum_{\jb\in \Lambda^{\circ}\cap \HH_n^{\circ} }  f(\tfrac{\jb}{4n})
   +12  \sum_{\jb\in \Lambda^{f}\cap \HH_n^{\circ}}  f(\tfrac{\jb}{4n}) \\
&  +6   \sum_{\jb\in \Lambda^{e,1}\cap \HH_n^{\circ}}f(\tfrac{\jb}{4n})
   +4   \sum_{\jb\in \Lambda^{e,2}\cap \HH_n^{\circ}}f(\tfrac{\jb}{4n})
   +   f(0),
\end{align*}
and, using the values of $c_\jb^{(n)}$,
\begin{align*}
&  \sum_{\jb\in\HH_n^{*} \setminus \HH_n^{\circ}} c_{\jb}^{(n)}f(\tfrac{\jb}{4n}) =
 24 \sum_{\jb\in \Lambda^{\circ}\cap \HH_n^{1,1} }  c_{\jb}^{(n)}f(\tfrac{\jb}{4n})
   + 12 \sum_{\jb\in \Lambda^{f}\cap \HH_n^{1,1}}  c_{\jb}^{(n)}f(\tfrac{\jb}{4n}) \\
& \qquad\qquad \qquad\qquad
+12  \sum_{\jb\in (\Lambda^{f}\cap \HH_n^{1,2})\cup
      (\Lambda^{f}\cap \HH_n^{2,1})} c_{\jb}^{(n)}f(\tfrac{\jb}{4n})
    + 6   \sum_{\jb\in \Lambda^{e,1}\cap \HH_n^{2,2}}  c_{\jb}^{(n)}f
(\tfrac{\jb}{4n}) \\
&  \qquad\qquad\qquad\qquad
  + 4   \sum_{\jb\in (\Lambda^{e,2}\cap \HH_n^{1,3})\cup (\Lambda^{e,2}
    \cap \HH_n^{3,1})} c_{\jb}^{(n)}f(\tfrac{\jb}{4n})\\
 &\qquad\qquad
  =  12  \sum_{\jb\in \Lambda^{\circ}\cap \HH_n^{1,1} }f(\tfrac{\jb}{4n})
   + 6  \sum_{\jb\in \Lambda^{f}\cap \HH_n^{1,1}}f(\tfrac{\jb}{4n})
 + 4  \sum_{\jb\in(\Lambda^{f}\cap \HH_n^{1,2})\cup
      (\Lambda^{f}\cap \HH_n^{2,1})}    f(\tfrac{\jb}{4n}) \\
&\qquad\qquad\qquad\qquad
   +    \sum_{\jb\in \Lambda^{e,1}\cap \HH_n^{2,2}}  f(\tfrac{\jb}{4n})
   +   \sum_{\jb\in (\Lambda^{e,2}\cap \HH_n^{1,3})\cup (\Lambda^{e,2}
         \cap \HH_n^{3,1})} f(\tfrac{\jb}{4n}).
\end{align*}
Adding these two expressions together and use \eqref{eq:Lambda-n2},
we conclude that
\begin{align} \label{eq:T2H}
     & \sum_{\jb\in \HH_n^*} c_{\jb}^{(n)}f(\tfrac{\jb}{4n})
  = \sum_{\jb\in \HH_n^{\circ}} f(\tfrac{\jb}{4n}) +
  \sum_{\jb\in\HH_n^{*} \setminus \HH_n^{\circ}} c_{\jb}^{(n)}f(\tfrac{\jb}{4n})
    = 24  \sum_{\jb\in \Lambda_n^{\circ} } f(\tfrac{\jb}{4n})   \\
 &\qquad\quad   +12  \sum_{\jb\in \Lambda_n^{f}} f(\tfrac{\jb}{4n})
   +6    \sum_{\jb\in \Lambda_n^{e,1}} f(\tfrac{\jb}{4n})
   +4    \sum_{\jb\in \Lambda_n^{e,2}} f(\tfrac{\jb}{4n})
   +      \sum_{\jb\in \Lambda_n^{v}}  f(\tfrac{\jb}{4n}) \notag \\
   & = \sum_{\jb\in \Lambda_n} \lambda_{\jb}^{(n)}f(\tfrac{\jb}{4n}). \notag
\end{align}

Replacing $f$ by $f\bar g$, we have proved that $\la f, g \ra_n^* =
\la f, g \ra_{\triangle,n}$ whenever $f \bar g$ is invariant. Hence,
\eqref{eq:orth-Tn} follows from Theorem \ref{ipdH}.
Furthermore, since  $\frac{1}{|\Omega_H|}\int_{\Omega_H} f(\tb) d\tb =
\frac{1}{|\triangle|_H}\int_{\triangle_H} f(\tb) d\tb$ for all invariant $f$,
\eqref{eq:cubature-T} follows from Theorem \ref{th:cubature-H}.

Furthermore, replacing $f$ by $\TC_{\kb}\overline{\TC_{\jb}}$ in  \eqref{eq:T2H},
we derive by \eqref{eq:TCs} that
\begin{align*}
   &\langle \TC_{\kb},\TC_{\jb} \rangle_{\triangle,n}
    = \frac{1}{4n^3}\sum_{\lb\in \HH_n^*}c_{\lb}^{(n)}\TC_{\kb}(\tfrac{\lb}{4n})\overline{\TC_{\jb}(\tfrac{\lb}{4n})}
\\ & = \frac{1}{4n^3}\frac{1}{24}\sum_{\lb\in \HH_n^*}c_{\lb}^{(n)}\sum_{\sigma\in \G}\phi_{\kb}(\tfrac{\lb \sigma}{4n}) \overline{\TC_{\jb}(\tfrac{\lb}{4n})}
     = \frac{1}{4n^3}\frac{1}{24}\sum_{\lb\in \HH_n^*}c_{\lb}^{(n)}\sum_{\sigma\in \G}\phi_{\kb}(\tfrac{\lb \sigma}{4n}) \overline{\TC_{\jb}(\tfrac{\lb\sigma}{4n})}
\\ & = \frac{1}{4n^3}\frac{|\G|}{24}\sum_{\lb\in \HH_n^*}c_{\lb}^{(n)}\phi_{\kb}(\tfrac{\lb}{4n}) \overline{\TC_{\jb}(\tfrac{\lb}{4n})}
     = \langle \phi_{\kb}, \TC_{\jb} \rangle_{n}^* = \frac1{24}\sum_{\sigma\in \G} \langle \phi_{\kb}, \phi_{\jb\sigma} \rangle_{n}^*.
\end{align*}
Using \eqref{k=jmod} and abbreviating $\kb \equiv \jb \mod 4 \ZZ_\HH^4$ as
$\kb \equiv \jb$, we further deduce that
\begin{align*}
 &\langle \TC_{\kb},\TC_{\jb} \rangle_{\triangle,n}
 = \frac1{24} \Big |\left \{\sigma\in \G:  \jb \sigma \equiv \kb \right\}\Big|
  =\frac{\delta_{\jb,\kb}}{24} \Big| \left\{\sigma\in \G:  \kb \sigma \equiv \kb \right\}\Big|
    = \frac{\delta_{\jb,\kb}}{\lambda^{(n)}_{\kb}},
\end{align*}
where the last equality follows from a direct counting.
\end{proof}

The proof of the above theorem also applies to $f, g \in \CTS_n$ since
$f \bar g $ is invariant if both $f$ and $g$ are anti-invariant. Moreover,
$f \bar g \in \CT_{2n-1}$ if $f, g \in \CT_n$. Notice also that
$\TS_{\kb}(\frac{\jb}{4n}) =0$ when $\jb \in  \Lambda_n\setminus
\Lambda_n^{\circ}$, we deduce the following result.

\begin{thm}
Let the discrete inner product $\langle \cdot, \cdot \rangle_{\triangle^{\circ},n}$
be defined by
\begin{align*}
\langle f, g \rangle_{\triangle^{\circ},n} = \frac{6}{n^3} \sum_{\jb \in \Lambda_n^{\circ}}
f(\tfrac{\jb}{4n}) \overline{g(\tfrac{\jb}{4n})}.
\end{align*}
Then
\begin{align*}
\langle f,g \rangle_{\triangle^{\circ},n} =  \langle f,g \rangle_{\triangle_H},\qquad
                f, g \in \CTS_{n}.
\end{align*}
\end{thm}


\subsection{Interpolation on the tetrahedron}

We can deduce results on interpolation on the tetrahedron by making use of the
orthogonality of generalized trigonometric functions with respect to the discrete
inner product, as shown in our first result below. Recall the operator
$\mathcal{P}^{\pm}$ defined in \eqref{eq:Ppm}.

\begin{thm} \label{thm:4.6}
For $n>0$, and $f\in C(\Delta_H)$, define
\begin{align*}
   \mathcal{L}_nf(\tb):= \sum_{\jb\in \Lambda_n^{\circ}}
    f(\tfrac{\jb}{4n}) \ell_{\jb,n}^{\circ}(\tb), \qquad
  \ell_{\jb,n}^{\circ} (\tb):= \frac{144}{n^3} \sum_{\kb\in \Lambda_n^{\circ}}
   \TS_{\kb}(\tb)    \overline{\TS_{\kb}(\tfrac{\jb}{4n})}.
\end{align*}
Then $\mathcal{L}_n f$ is the unique function in $\CTS_n$ that satisfies
\begin{align*}
\mathcal{L}_nf( \tfrac{\jb}{4n}) =f( \tfrac{\jb}{4n}), \qquad \jb\in \Lambda_n^{\circ}.
\end{align*}
Furthermore, the fundamental interpolation function $\ell_{\jb,n}^{\circ}$ is real and satisfies
\begin{align*}
   \ell_{\jb,n}^{\circ} (\tb) = \frac{6}{n^3} \mathcal{P}^{-}_{\tb} \left[
   \Theta_{n}(\tb-\tfrac{\jb}{4n}) - \Theta_{n-1}(\tb-\tfrac{\jb}{4n})
   \right],
\end{align*}
where $\mathcal{P}^{-}_{\tb}$ means that the operator $\mathcal{P}^{-}$ is
acting on the variable $\tb$ and $\Theta_{n}$ is defined in \eqref{DnTheta-n}.
\end{thm}

\begin{proof}
By \eqref{eq:orth-TS}, $\langle \TS_{\jb},\TS_{\kb} \rangle_{\triangle^{\circ},n} =
\frac{1}{24} \delta_{\jb,\kb}$, which shows that $\ell_{\jb,n}^{\circ}(\frac{\kb}{4n})
= \delta_{\jb,\kb}$ and verifies the interpolation condition.
It follows from the definition of $\TS_{\kb}$ that
\begin{align*}
\ell_{\jb,n}^{\circ} (\tb) = \frac{6}{n^3}\mathcal{P}^{-}_{\tb}\mathcal{P}^{-}_{\jb}
\sum_{\kb\in \HH_n^{\circ}} \phi_{\kb}(\tb)\overline{\phi_{\kb}(\tfrac{\jb}{4n})}.
\end{align*}
Furthermore, we can replace the summation over $\kb \in \HH_n^\circ$ by
the summation over $\kb \in \HH_n$ since $\TS_{\kb}(\tfrac{\jb}{4n})=
\TS_{\jb}(\tfrac{\kb}{4n}) = 0$ whenever
$\kb\in \Lambda_n\setminus \Lambda_n^{\circ}$, and $\TS_{\kb}(\tb)=0$ whenever
at least two components of $\kb$ are equal. Consequently,  we conclude that
\begin{align*}
\ell_{\jb,n}^{\circ} (\tb) = \frac{6}{n^3}\mathcal{P}^{-}_{\tb}\mathcal{P}^{-}_{\jb}
 D_n^H(\tb-\tfrac{\jb}{4n}),
 \end{align*}
where $D_n^H$ is the Dirichlet kernel for the rhombic dodecahedral Fourier partial
sum defined in \eqref{def:Dn^H}. Recall that $\mathcal{G}$ is a permutation group
and $|\sigma|$ denote the number of inversions in $\sigma \in \G$.  Let $f$ be an
invariant function under $\G$. Then
 \begin{align*}
 \mathcal{P}^{-}_{\tb}\mathcal{P}^{-}_{\sb} f(\tb-\sb)
& =  \frac{1}{|\mathcal{G}|^2}\sum_{\sigma \in \mathcal{G}}
      \sum_{\tau \in \mathcal{G}} (-1)^{|\sigma|+| \tau|}
        f(\tb \sigma  -  \sb \tau) \\
& = \frac{1}{|\mathcal{G}|^2}\sum_{\tau \in \mathcal{G}}\sum_{\sigma \in \mathcal{G}}
   (-1)^{|\sigma|+ |\tau|}  f( \tb \sigma \tau^{-1} -  \sb ) \\
& =\frac{1}{|\mathcal{G}|}\sum_{\sigma \in \mathcal{G}} (-1)^{|\sigma|}
   f( \tb \sigma -  \sb ) = \mathcal{P}^{-}_{\tb} f( \tb -  \sb ),
 \end{align*}
where in the third equal sign we have used the fact that $|\sigma \tau| + |\tau|
= |\sigma|$, which can be easily verified. Setting $f = D_n^H$ completes the
proof.
 \end{proof}

The function $\CL_n f$ interpolates at the interior points of $\Lambda_n$. We can
also derive an analog result for interpolation on $\Lambda_n$ by using the same
approach. However, it is more illustrating to derive it from the interpolation on the
rhombic dodecahedron, which we carry out below.

\begin{thm} \label{thm:4.7}
For $n>0$ and $f\in C(\Delta_H)$ define
\begin{align*}
  \mathcal{L}^*_n f(\tb):= \sum_{\jb\in \Lambda_n} f(\tfrac{\jb}{4n})
     \ell^{\triangle}_{\jb,n}(\tb),
 \qquad  \ell^{\triangle}_{\jb,n}(\tb) := \frac{\lambda_{\jb}^{(n)}}{4n^3}
 \sum_{\kb\in \Lambda_n}\lambda^{(n)}_{\kb} \TC_{\kb}(\tb)
       \overline{\TC_{\kb}(\tfrac{\jb}{4n})}.
\end{align*}
Then $\mathcal{L}^*_n f$ is the unique function in $\CTC_n$ that satisfies
\begin{align*}
\mathcal{L}^*_n f(\tfrac{\jb}{4n}) = f(\tfrac{\jb}{4n}), \qquad \jb\in \Lambda_n.
\end{align*}
Furthermore, the fundamental interpolation function $\ell^{\triangle}_{\jb,n}$
is given by
\begin{align*}
\ell^{\triangle}_{\jb,n}(\tb) = \lambda_{\jb}^{(n)} \mathcal{P}^{+}
\ell_{\jb,n}(\tb).
\end{align*}
\end{thm}

\begin{proof}
It follows from \eqref{eq:orth-TCn} that
$\ell^{\triangle}_{\jb,n}(\frac{\kb}{4n}) =\delta_{\kb,\jb}$ for
$\kb,\jb\in \Lambda_n$, which verifies the interpolation condition.
Furthermore, in the proof of Theorem \ref{thm:4.4}, we established
that $\sum_{\jb \in \Lambda_n}\lambda_\jb^{(n)} g(\frac{\jb}{4n}) =
 \sum_{j\in \HH_n^*}c_\jb^{(n)} g(\frac{\jb}{4n})$ for function $g$ invariant under $\G$.
 Applying this relation to $g(\frac{\kb}{4n}) = \TC_\kb(\tb) \TC_\kb(\frac{\jb}{4n})$, we obtain
\begin{align*}
\ell_{\jb,n}^{\triangle} (\tb) =& \frac{\lambda_{\jb}^{(n)}}{4n^3}
\sum_{\kb\in \HH_n^*}c^{(n)}_{\kb}
   \TC_{\kb}(\tb)   \overline{ \TC_{\kb}(\tfrac{\jb}{4n})}\\
 = & \frac{\lambda_{\jb}^{(n)}}{4n^3} \mathcal{P}^{+}_{\tb}\mathcal{P}^{+}_{\jb}
\sum_{\kb\in \HH_n^*}
c^{(n)}_{\kb}\phi_{\kb}(\tb)\overline{\phi_{\kb}(\tfrac{\jb}{4n})} =
\lambda_{\jb}^{(n)}\mathcal{P}^{+}_{\tb}\mathcal{P}^{+}_{\jb}
  \Phi^{*}_{n}(\tb-\tfrac{\jb}{4n}).
\end{align*}
Using the fact that $\mathcal{G}$ is a permutation group, it is easy to see that
\begin{align*}
 \mathcal{P}^{+}_{\tb}\mathcal{P}^{+}_{\sb} f(\tb-\sb)
 = \mathcal{P}^{+}_{\tb} f( \tb -  \sb)
 \end{align*}
for an invariant function $f$. Consequently,
\begin{align*}\ell_{\jb,n}^{\triangle} (\tb)=  \lambda_{\jb}^{(n)}
\mathcal{P}^{+}_{\tb} \Phi^{*}_{n}(\tb-\tfrac{\jb}{4n}) =
\lambda_{\jb}^{(n)} \mathcal{P}^{+} \ell_{\jb,n}(\tb).
\end{align*}
The proof is completed.
\end{proof}

Recall the explicit formula of $\ell_{\jb,n}$ given in Theorem \ref{interp-H},
$\ell_{j,n}^\triangle$ enjoys a compact formula.

Let $\|\mathcal{L}_n\|$ and $\|\mathcal{L}_n^*\|$ denote the
operator norms of $\mathcal{L}_n$ and $\mathcal{L}_n^*$,
respectively, both as operators from $C(\triangle_H)\mapsto
C(\triangle_H)$. From Theorems \ref{thm:4.6} and \ref{thm:4.7}, an
immediate application of Theorem \ref{LebesgueH} yields the
following theorem.

\begin{thm}
There is a constant $c$ independent of $n$, such that
\begin{align*}
  \|\mathcal{L}_n\| \leq c (\log n)^3 \quad\hbox{and}\quad
    \|\mathcal{L}_n^*\| \leq c (\log n)^3.
\end{align*}
\end{thm}


\subsection{Interpolation on the regular tetrahedron}

The results in the above are developed in homogeneous coordinates. Here
we indicate how they can be recast into the usual coordinates on the regular
tetrahedron $\triangle^*$ defined by
$$
    \triangle^*: = \left\{ x\in \RR^3: 0 \le x_3\le x_2\le x_1\le 1  \right\}
$$
as depicted in the Figure 4.3 below.

\begin{figure}[htb]
\centering
\includegraphics[width=0.6\textwidth]{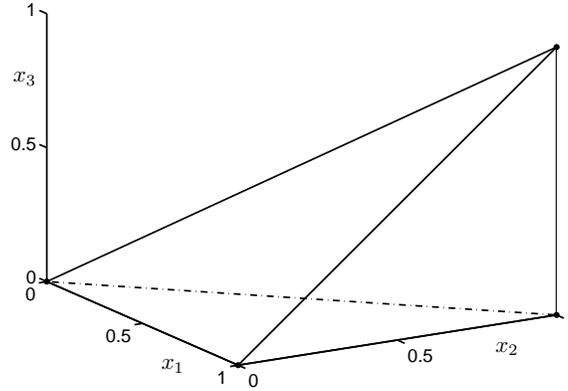}
\caption{Regular tetrahedron}
\end{figure}

The change of variable from $\tb$ to $x \in \RR^3$ is given in \eqref{coordinate}.
When we transform the formulas from the homogeneous coordinates to the
regular coordinates, we also need to transform the indices from $\jb \in \ZZ_H^4$
to $\ZZ^3$ by using
\begin{align} \label{jb-k}
        k = \frac14 A^\tr U^\tr \jb  \quad \Longleftrightarrow \quad
     \begin{cases} k_1=   \tfrac14 (j_1 - j_4)  \\
           k_2 =   \tfrac14 (j_2 - j_4)  \\
           k_3 =   \tfrac14 (j_3 - j_4).
      \end{cases}
\end{align}
Under this change variables, it is easy to see that the point set $\Lambda_n$
becomes $\{k \in \ZZ^3: 0 \le k_3 \le k_2\le k_1 \le n\}$. For example, the
cubature formula in \eqref{eq:cubature-T} becomes the following:

\begin{thm}
For $n>0$, the cubature formula
\begin{align*}
 6\int_{\triangle^*} f(x_1,x_2,x_3)dx_1dx_2dx_3 = \frac{1}{4n^3}
  \sum_{0\leq k_3\leq k_2\leq k_1\leq n}
 \lambda_{k_1,k_2,k_3}^{(n)} f(\tfrac{k_1}{n},\tfrac{k_2}{n},\tfrac{k_3}{n})
\end{align*}
is exact for all $f\in \CTC_{2n-1}$, where
$\lambda_{k_1,k_2,k_3}^{(n)} =  \lambda_\jb^{(n)}$
with $\jb$ given by \eqref{jb-k}.
\end{thm}

We also note that the Dirichlet kernel in \eqref{DnTheta-n} can be recast
into $x$ coordinates by \eqref{coordinate} straightforwardly, so is the
fundamental interpolation function $\ell_{\jb,n}(\tb)$ given in Theorem
\ref{interp-H}. Consequently, the Lagrange interpolation function in
Theorem \ref{thm:4.7} becomes
$$
  \CL_n^* f(x) = \sum_{0\le k_3\le k_2\le k_1 \le n}
     f (\tfrac{k_1}{n},\tfrac{k_2}{n},\tfrac{k_3}{n}) \ell_{k,n}^\triangle(x),
$$
where $\ell_{k,n}^\triangle$ satisfies the compact formula
$$
   \ell_{k,n}^\triangle (x) = \lambda_k^{(n)} \CP^+ \ell_{\jb,n}(\tb),
        \quad \hbox{where} \quad
       \ell_{\jb,n}(\tb) = \Phi_n^*(\tb - \tfrac{\jb}{4n})
$$
with $\Phi_n^*$ given in \eqref{eq:Phi-n*}, $\lambda_k^{(n)} = \lambda_\jb^{(n)}$
with $\jb$ as in \eqref{jb-k}. In the above formula we apply $\CP^+$ to the compact
 formula of $\Phi_n^*$ first and then use \eqref{coordinate} to change from $\tb$
 to $x$.

\bigskip\noindent
{\it Acknowledgment}. The authors thank an anonymous referee for his careful
reading and invaluable comments.


\begin{thebibliography}{99}

\bibitem{Bos}
        L. Bos,
        Bounding the Lebesgue function for Lagrange interpolation in a simplex,
        \textit{J. Approx. Theory} \textbf{38} (1983), 43--59.

\bibitem{CS}
        J. H. Conway and N. J. A. Sloane,
        \textit{Sphere Packings, Lattices and Groups},  3rd ed.
        Springer, New York, 1999.

\bibitem{DM}
        D. E. Dudgeon and R. M. Mersereau,
        \textit{Multidimensional Digital Signal Processing},
        Prentice-Hall Inc, Englewood Cliffs, New Jersey, 1984.

\bibitem{DX}
        C. F. Dunkl and Yuan Xu,
        \textit{Orthogonal polynomials of several variables},
        Encyclopedia of Mathematics and its Applications, vol. {\bf 81},
        Cambridge Univ. Press, 2001.

\bibitem{F}
        B.  Fuglede,
        Commuting self-adjoint partial differential operators
        and a group theoretic problem,
        \textit{J. Functional Anal.} \textbf{16} (1974), 101--121.

\bibitem{H}
        T. C. Hales,
        A proof of the Kepler conjecture.
        \textit{Ann. of Math.} (2) \textbf{162} (2005), no. 3, 1065--1185.

\bibitem{Hi}
        J. R. Higgins,
        \textit{Sampling theory in Fourier and Signal Analysis, Foundations},
       Oxford Science Publications, New York, 1996.

\bibitem{K}
        T. Koornwinder,
        Orthogonal polynomials in two variables which are eigenfunctions
        of two algebraically independent partial differential operators,
        \textit{Nederl. Acad. Wetensch. Proc. Ser. A77} = \textit{Indag. Math}.
        \textbf{36} (1974), 357-381.

\bibitem{LSX}
        H. Li, J. Sun and Yuan Xu,
        Discrete Fourier analysis, cubature and interpolation on a hexagon
        and a triangle,
        \textit{SIAM J. Numer. Anal.}, to appear (preprint, 2006).

\bibitem{Ma}
        R. J. Marks II,
        \textit{Introduction to Shannon Sampling and Interpolation Theory},
        Springer-Verlag, New York, 1991.

\bibitem{Sun}
       J. Sun,
       Multivariate Fourier series over a class of non tensor-product
       partition domains,
       \textit{J. Comput. Math.} \textbf{21} (2003), 53-62.

\bibitem{Sun2}
       J. Sun,
       Multivariate Fourier transform methods over simplex and super-simplex
       domains,
       \textit{J. Comput. Math.} \textbf{24} (2006), 305-322.

\bibitem{Z}
        A. Zygmund,
        \textit{Trigonometric series},
        Cambridge Univ. Press, Cambridge, 1959.

\end{thebibliography}
\end{document}